%% file: Perceptron-v2.tex
\documentclass[11pt]{article}

\usepackage[margin=1.05in]{geometry}
\usepackage{amsmath,amssymb,amsthm,mathtools}
\usepackage[T1]{fontenc}
\usepackage{lmodern}
\usepackage{microtype}
\usepackage{tikz}
\usepackage{float}
\usepackage{caption}
\usepackage{subcaption}
\usepackage{needspace}
\usetikzlibrary{arrows.meta,calc}
\usepackage[colorlinks=true,linkcolor=blue,citecolor=blue,urlcolor=blue]{hyperref}
\hypersetup{
  pdftitle={A Robust Perceptron Cycling Theorem and Applications},
  pdfauthor={Tobias Harks}
}

\newtheorem{theorem}{Theorem}
\newtheorem{lemma}{Lemma}
\newtheorem{proposition}{Proposition}
\newtheorem{corollary}{Corollary}
\newtheorem{claim}{Claim}
\theoremstyle{definition}
\newtheorem{definition}{Definition}
\theoremstyle{remark}
\newtheorem{remark}{Remark}

\newcommand{\R}{\mathbb{R}}
\newcommand{\ip}[2]{\langle #1,#2\rangle}
\newcommand{\norm}[1]{\lVert #1\rVert}
\newcommand{\Sph}{\mathcal{S}}
\DeclareMathOperator{\conv}{conv}
\DeclareMathOperator{\Id}{Id}
\DeclareMathOperator*{\argmin}{arg\,min}

\input{Proof_Overview_Combined-v2.tikz}
\input{Cube_Three_Regimes-v2.tikz}

\title{A Robust Perceptron Cycling Theorem and Applications}
\author{Tobias Harks\thanks{Chair of Mathematical Optimization,
University of Passau, Dr.-Hans-Kapfinger-Str.~30, 94032 Passau, Germany.
Email: \href{mailto:tobias.harks@uni-passau.de}{\texttt{tobias.harks@uni-passau.de}}.}}
\date{\today}

\begin{document}
\maketitle

\begin{abstract}
The classical perceptron cycling theorem of Block and Levin \cite{BlockLevin1970} bounds correction sequences whose selected updates
come from a finite set and have nonpositive inner product with the current
state.  We prove a robust variant for additive trajectories
$z_{k+1}=z_k+u_k$, for all integers $k\geq0$, with increments in a finite
set $U\subset E$, where $E$ is a finite-dimensional real inner-product
space and $0\in\conv U$.  Let $A\colon E\to E$ have positive-definite
symmetric part, without requiring symmetry, and let $B\geq0$.
If each $u_k$ is a $B$-approximate
minimizer of $u\mapsto\ip{Az_k}{u}$ over $U$, then
$\sup_{k\geq0}\norm{z_k}\leq C(1+\norm{z_0}+B)$, with
$C=C(E,U,A)$ independent of the initial state, $B$, and all admissible
update choices.  We derive two algorithmic consequences.  The first is an
$O(k^{-1})$ last-iterate norm bound for harmonic vertex-returning
Frank--Wolfe for affine strongly monotone variational inequalities on
polytopes with relatively interior solutions.  It extends the quadratic
Frank--Wolfe/herding guarantee of Bach, Lacoste-Julien, and Obozinski
\cite[Section~4.2]{BachEtAl2012} to nonsymmetric affine operators.
The second consequence concerns oblique relaxation for linear inequalities.  Greedy
corrections through a fixed matrix with positive-definite symmetric part
remain bounded even for inconsistent systems.  In particular, for a
square matrix $G$ with positive-definite symmetric part, repeatedly
increasing the coordinate corresponding to a most-violated inequality
of $Gx\geq b$ terminates at an exactly feasible point after finitely many
unit corrections.  This remains true under bounded additive selection
errors, provided the stopping test uses the true inequalities.
\end{abstract}

\section{Introduction}\label{sec:introduction}

The binary linear classification problem underlying Rosenblatt's
perceptron \cite{Rosenblatt1958} asks whether labeled data
$(x_i,y_i)_{i=1}^m$, with $x_i\in\R^p$ and $y_i\in\{-1,1\}$, can be
separated according to their labels.  Put
$a_i:=y_i x_i$.  A vector $w\in\R^p$ defines a homogeneous classifier with
strictly positive margin on the data when
\[
  \ip{w}{a_i}=y_i\ip{w}{x_i}>0
  \qquad\text{for all }i=1,\ldots,m.
\]
Such a vector exists if and only if
$0\notin\conv\{a_1,\ldots,a_m\}$, by the strict finite-dimensional
separation theorem in Rockafellar \cite[Section~11]{Rockafellar1970}.  Affine classifiers
are included by appending a bias coordinate to each feature vector.

The perceptron correction algorithm starts from a weight vector $w_0$.  At
step $k$, if some example is misclassified or has zero margin, it chooses an
index $i_k$ satisfying
\[
  \ip{w_k}{a_{i_k}}\leq0
\]
and performs the correction
\[
  w_{k+1}=w_k+a_{i_k}.
\]
For strictly separable data the classical perceptron convergence theorem
ensures that only finitely many corrections are made; see
Novikoff \cite{Novikoff1962} and Minsky and Papert \cite{MinskyPapert1969},
or the modern presentation by Shalev-Shwartz and Ben-David
\cite[Theorem~9.1]{ShalevShwartzBenDavid2014}.  For nonseparable data
the procedure need not terminate, so the relevant question becomes whether
its weight sequence can escape to infinity.
The classical perceptron cycling theorem of Block and Levin
\cite{BlockLevin1970} answers this question.  Let $U$ be a finite subset of a
Euclidean space with $0\in\conv U$.  Every sequence satisfying, for all
integers $k\geq0$,
\begin{equation}\label{eq:classical-cycling-intro}
  z_{k+1}=z_k+u_k,
  \qquad u_k\in U,
  \qquad \ip{z_k}{u_k}\leq0
\end{equation}
obeys
\begin{equation}\label{eq:classical-bound-intro}
  \norm{z_k}\leq\norm{z_0}+M(U)
  \qquad\text{for all integers }k\geq0,
\end{equation}
where $M(U)$ is independent of the initial point and of all admissible update
choices.  In the classification problem, nonseparability is equivalent to
$0\in\conv U$ for $U:=\{a_1,\ldots,a_m\}$.  Taking $z_k=w_k$ and
$u_k=a_{i_k}$ identifies the perceptron corrections with
\eqref{eq:classical-cycling-intro}.  The theorem therefore shows that the
weights remain bounded even when the correction procedure runs indefinitely.
It also gives the averaging estimate
\begin{equation}\label{eq:classical-average-intro}
  \frac1T\sum_{k=0}^{T-1}u_k
  =\frac{z_T-z_0}{T}=O(T^{-1}),
\end{equation}
where $T$ denotes a
fixed terminal horizon.
Thus boundedness of the iterates also forces the average correction vector
to converge to zero at rate $O(T^{-1})$.

The purpose of this paper is to establish a \emph{robust} perceptron cycling theorem
for update rules in which the increment is chosen by approximate global
minimization of a linearly transformed score. 
Let $E$ be a finite-dimensional real
inner-product space, $U\subset E$ a nonempty finite set with
$0\in\conv U$, and $A\colon E\to E$ a linear map such that
\begin{equation}\label{eq:strong-positive}
  \ip{Ax}{x}\geq\mu\norm{x}^2
  \qquad\text{for all }x\in E \qquad (\text{coercivity})
\end{equation}
for some $\mu>0$.  We write $\mathbb N_0:=\{0,1,2,\ldots\}$.
\begin{definition}[trajectory and $B$-approximate trajectory]
\label{def:B-trajectory}
A \emph{$U$-valued trajectory} consists of a state sequence
$(z_k)_{k\in\mathbb N_0}$ in $E$ together with increments
$(u_k)_{k\in\mathbb N_0}$ in $U$ satisfying
\begin{equation}\label{eq:additive-trajectory}
  z_{k+1}=z_k+u_k
  \qquad\text{for all }k\in\mathbb N_0.
\end{equation}
The elements $z_k$ are its \emph{states} and the elements $u_k$ are its
\emph{increments}.  Because $u_k=z_{k+1}-z_k$, we usually identify a
trajectory with its state sequence $(z_k)_{k\in\mathbb N_0}$.
For a nonzero state $z$, its \emph{unit direction} is the vector
$z/\norm z$.  When we say that a trajectory enters a set of unit
directions, we mean that the unit direction of one of its nonzero states
belongs to that set.

Let $B\geq0$.  A $U$-valued trajectory is a
\emph{$B$-approximate trajectory for $(E,U,A)$} if, for every
$k\in\mathbb N_0$,
\begin{equation}\label{eq:trajectory}
  \ip{Az_k}{u_k}
  \leq\min_{u\in U}\ip{Az_k}{u}+B.
\end{equation}
A $0$-approximate trajectory is also called exact.
For integers $0\leq k_0\leq k_1$, a \emph{$U$-valued trajectory segment}
consists of states $(z_k)_{k_0\leq k\leq k_1}$ in $E$ and increments
$u_k\in U$ satisfying $z_{k+1}=z_k+u_k$ for all integers $k$ with
$k_0\leq k<k_1$.  It is
\emph{$B$-approximate} if
\[
  \ip{Az_k}{u_k}
  \leq\min_{u\in U}\ip{Az_k}{u}+B
  \qquad\text{for all integers }k\text{ with }k_0\leq k<k_1.
\]
\end{definition}
Compared with the classical Euclidean score $\ip{z_k}{u}$, we use the
transformed score $\ip{Az_k}{u}$, where $A$ may be nonsymmetric but its
symmetric part is positive definite by \eqref{eq:strong-positive}.  The
selection rule is nevertheless more specific than the one in the classical
cycling theorem: that theorem allows any update with nonpositive score,
whereas Definition~\ref{def:B-trajectory} requires an approximate global
minimizer.  If $A=\Id_E$ and $B=0$, every exact minimizer in
Definition~\ref{def:B-trajectory} has nonpositive score because
$0\in\conv U$, and is therefore covered by the classical theorem of Block and Levin
\cite{BlockLevin1970}.
The word \emph{robust} refers to the additive tolerance $B$.  The selected
increment need only be within $B$ of the optimal transformed score, and the
resulting state bound deteriorates at most linearly with $B$, uniformly over
all such trajectories.

\subsection{Our main result}\label{subsec:main-result-intro}
The new theorem states that every trajectory from
Definition~\ref{def:B-trajectory} remains in a ball whose radius grows at
most linearly with the initial norm and the oracle error.
For \emph{symmetric} positive-definite $A$, such a bound already follows from
Amaldi and Hauser's relaxation theorem
\cite[Theorem~7.1]{AmaldiHauser2005}; the reduction is given in
Section~\ref{sec:related}.  Our contribution is to establish a
uniform bound for general coercive $A$, without symmetry or an additional
restriction on its skew-symmetric part.
\begin{theorem}[Robust perceptron cycling]\label{thm:main}
There is a constant $C=C(E,U,A)\geq1$ such that, for every $B\geq0$, every
$B$-approximate $U$-valued trajectory satisfies
\begin{equation}\label{eq:main-bound}
  \sup_{k\geq0}\norm{z_k}
  \leq C\bigl(1+\norm{z_0}+B\bigr).
\end{equation}
\end{theorem}
Let us give a short example illustrating the theorem and its boundaries.
Figure~\ref{fig:cube-regimes} compares three exact trajectories with
$U=\{-1,1\}^2$, $z_0=(3/4,2/5)$, and $B=0$ illustrating the effect of
the linear map $A$ on the resulting trajectory.
\begin{figure}[H]
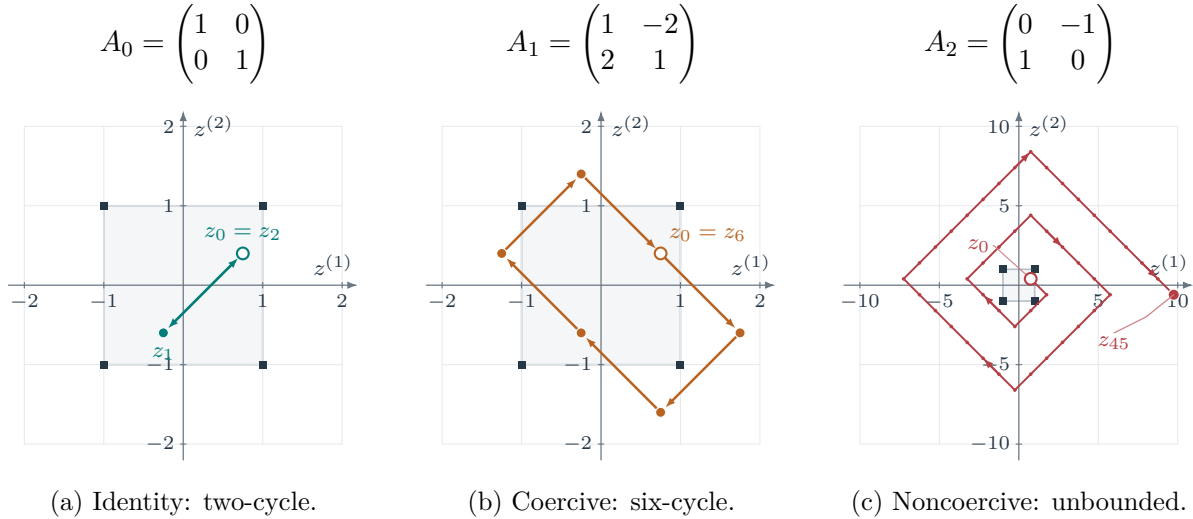

\centering
\captionsetup[subfigure]{font=small,justification=centering}
\begin{subfigure}[t]{.32\textwidth}
\centering
$\displaystyle A_0=\begin{pmatrix}1&0\\0&1\end{pmatrix}$
\par\smallskip
\resizebox{\linewidth}{!}{\CubeTrajectoryPanel{0}}
\caption{Identity: two-cycle.}\label{fig:cube-identity}
\end{subfigure}\hfill
\begin{subfigure}[t]{.32\textwidth}
\centering
$\displaystyle A_1=\begin{pmatrix}1&-2\\2&1\end{pmatrix}$
\par\smallskip
\resizebox{\linewidth}{!}{\CubeTrajectoryPanel{1}}
\caption{Coercive: six-cycle.}\label{fig:cube-coercive}
\end{subfigure}\hfill
\begin{subfigure}[t]{.32\textwidth}
\centering
$\displaystyle A_2=\begin{pmatrix}0&-1\\1&0\end{pmatrix}$
\par\smallskip
\resizebox{\linewidth}{!}{\CubeTrajectoryPanel{2}}
\caption{Noncoercive: unbounded.}\label{fig:cube-skew}
\end{subfigure}
\caption{Exact dynamics for the matrices shown above.  The shaded squares
are $\conv U$, not constraints on the states; black squares mark the
allowed increments.  Subfig. (a) and (b) have the same scale and symmetric
part $I$.  Subfig. (c) uses a wider scale and shows the first 45 updates of
an unbounded trajectory.  The cycles illustrate boundedness; periodicity
need not generally be true.}
\label{fig:cube-regimes}
\end{figure}
The skew-symmetric example in Figure~\ref{fig:cube-skew} leads to an unbounded trajectory
even with $B=0$ and every oracle minimizer being unique.
Coercivity cannot in general be replaced by mere monotonicity, that is,
$\ip{Ax}{x}\geq0$ for all $x\in E$.
Subsection~\ref{subsec:skew-counterexample} verifies this divergence and
gives a further example with $\norm{z_k}=\Theta(\sqrt{k})$.
The tolerance $B$ also represents imperfect information in the selection
rule.  Bounded delays and uniformly bounded state, quantization, or score
errors give a uniformly bounded error in the linear scores.  Approximate
minimization of these scores up to a fixed additive error therefore
produces a $B$-approximate trajectory for a fixed $B$, so
Theorem~\ref{thm:main} still guarantees
boundedness.  Proposition~\ref{prop:robustness-mechanisms} in
Appendix~\ref{app:oracle-errors} gives the combined error bound.
The conversion from these errors to $B$ depends on the application;
Subsections~\ref{subsec:affine-fw-app}
and~\ref{subsec:oblique-relaxation} explain the relevant distinctions.

A rounding argument gives the following extension to bounded nonnegative
update weights.  The oracle still selects a direction from the fixed
finite set $U$.
\begin{corollary}[Weighted updates]\label{cor:weighted}
Under the standing assumptions on $E$, $U$, and $A$, there is a constant
$C_{\mathrm w}=C_{\mathrm w}(E,U,A)\geq1$ such that the following holds.
Let $\overline w\geq0$ and $B\geq0$.  Suppose that, for every $k\in\mathbb N_0$,
\begin{equation}\label{eq:weighted-recursion}
  z_{k+1}=z_k+w_ku_k,
  \qquad 0\leq w_k\leq\overline w,\qquad u_k\in U,
\end{equation}
and
\begin{equation}\label{eq:weighted-oracle}
  \ip{Az_k}{u_k}\leq\min_{u\in U}\ip{Az_k}{u}+B.
\end{equation}
Then
\begin{equation}\label{eq:weighted-bound}
  \sup_{k\geq0}\norm{z_k}
  \leq C_{\mathrm w}\bigl(\overline w+\norm{z_0}+B\bigr).
\end{equation}
The same bound holds for any finite recursion satisfying these conditions
at every update, with the supremum taken over its state indices.
No monotonicity or positive lower bound on the weights is required.
\end{corollary}

Telescoping the weighted update equation gives an averaging estimate in
terms of the total weight.
\begin{corollary}[Weighted averaging]\label{cor:averaging}
Under the assumptions of Corollary~\ref{cor:weighted}, define
$\Lambda_T:=\sum_{k=0}^{T-1}w_k$ for every integer $T\geq1$.
For every such $T$ with $\Lambda_T>0$,
\begin{equation}\label{eq:averaging-consequence}
  \norm{\frac1{\Lambda_T}\sum_{k=0}^{T-1}w_ku_k}
  =\frac{\norm{z_T-z_0}}{\Lambda_T}
  \leq
  \frac{C_{\mathrm w}(\overline w+\norm{z_0}+B)+\norm{z_0}}{\Lambda_T}.
\end{equation}
If $\lim_{T\to\infty}\Lambda_T=\infty$, the weighted averages therefore converge to zero
at rate $O(\Lambda_T^{-1})$.  This gives an $O(T^{-1})$ bound if there are
$c>0$ and an integer $T_0\geq1$ such that $\Lambda_T\geq cT$ for every
integer $T\geq T_0$.
For unit weights $w_k=1$ for all $k\in\mathbb N_0$, the constant $C$ from
Theorem~\ref{thm:main} gives the bound
\[
  \norm{\frac1T\sum_{k=0}^{T-1}u_k}
  \leq\frac{C(1+\norm{z_0}+B)+\norm{z_0}}T
  \qquad\text{for every integer }T\geq1.
\]
\end{corollary}

\begin{proof}
The identity follows from
$\sum_{k=0}^{T-1}w_ku_k=z_T-z_0$; the estimate follows from
Corollary~\ref{cor:weighted} and the triangle inequality.  For unit
weights, apply Theorem~\ref{thm:main} instead.
\end{proof}

The weighted bound also has a vertex-returning polytope form, including
approximate linear minimization.

\begin{corollary}[Vertex-returning polytope form]\label{cor:polytope}
Let $V\subset\R^D$ be a polytope with $0\in V$, let
$U:=\operatorname{vert}(V)$, let $A\colon\R^D\to\R^D$ be linear, and suppose that
$\tfrac12(A+A^{\mathsf T})\succ0$.  Let $\overline w,B\geq0$.  If, for every
$k\in\mathbb N_0$,
\[
  z_{k+1}=z_k+w_ku_k,
  \qquad 0\leq w_k\leq\overline w,\qquad u_k\in U,
\]
and
\[
  \ip{Az_k}{u_k}\leq\min_{v\in V}\ip{Az_k}{v}+B,
\]
then
\[
  \sup_{k\geq0}\norm{z_k}
  \leq C_{\mathrm w}(\R^D,U,A)\bigl(\overline w+\norm{z_0}+B\bigr).
\]
In particular, exact vertex minimizers
$u_k\in U\cap\argmin_{v\in V}\ip{Az_k}{v}$ are covered with $B=0$.
\end{corollary}

\begin{proof}
Since $V=\conv U$, the assumption $0\in V$ gives $0\in\conv U$.
A linear functional has the same minimum over $V$ and over its vertex set
$U$.  The oracle inequality is therefore precisely
\eqref{eq:weighted-oracle}, and Corollary~\ref{cor:weighted} applies.
\end{proof}

To complement the qualitative bound of Theorem~\ref{thm:main} with an
explicit constant, we next impose symmetry and relative interiority.
Amaldi and Hauser already give an explicit bound; see the comparison in
Section~\ref{sec:related}.  Under these stronger assumptions, a separate
quadratic Lyapunov argument yields the relative-inradius estimate below.
This estimate can be substantially smaller than the bound obtained by a
direct application of Amaldi and Hauser's general formula; see the
example in Appendix~\ref{app:bound-comparison}.

\begin{theorem}[Explicit symmetric relative-interior bound]
\label{thm:explicit-symmetric}
Suppose that $U\neq\{0\}$, $A$ is symmetric positive definite, and
$0\in\operatorname{relint}(\conv U)$.  Write $\lambda_{\min}$ and
$\lambda_{\max}$ for the smallest and largest eigenvalues of $A$.
Set $V:=A^{1/2}U$ and $R:=\max_{v\in V}\norm v$, where $A^{1/2}$ is
the symmetric positive-definite square root of $A$.  Let $r>0$ be the
relative inradius of $\conv V$ at zero: the radius of the largest ball
centered at zero in $\operatorname{span}V$ and contained in $\conv V$.
Then, for every $B\geq0$, every $B$-approximate trajectory satisfies
\begin{equation}\label{eq:explicit-symmetric-simple}
 \sup_{k\geq0}\norm{z_k}
 \leq
 \sqrt{\frac{\lambda_{\max}}{\lambda_{\min}}}\,\norm{z_0}
 +\frac1{\sqrt{\lambda_{\min}}}
  \left(R+\frac{R^2}{2r}+\frac Br\right).
\end{equation}
\end{theorem}
Subsection~\ref{subsec:proof-explicit-symmetric} provides a formal proof of this bound.

\subsection{Proof technique}\label{subsec:proof-technique-intro}
The proof of Theorem~\ref{thm:main} proceeds by induction on the dimension
of $E$.  If the theorem failed in the inductive step, there would be
approximate trajectories reaching target radii arbitrarily large relative
to $1+\norm{z_0}+B$.  We combine an estimate obtained by face reduction
near the \emph{degenerate directions}, the unit vectors at which
the minimum transformed score is zero, with a decrease estimate for
trajectory segments outside a neighborhood of these directions.

For each degenerate direction $d$, the increments minimizing the score at
$d$ lie in a face of $\conv U$ contained in the lower-dimensional subspace
$(Ad)^\perp$.  While the selected increments belong to this face,
projection onto this subspace along $d$ gives a $B$-approximate trajectory.
The difference between the original and projected states stays constant.
The induction hypothesis bounds the projected trajectory.  We use this
bound to prove that, if a state has unit direction sufficiently close to
$d$ and norm above a prescribed threshold, every subsequent increment
remains in the same face.  This yields a bound on all subsequent states
in terms of the entrance norm and the oracle error.  The neighborhood
and the constants initially depend on the chosen direction.  Compactness
of the set of degenerate directions allows us to cover it by finitely many
of these neighborhoods.  Their union defines a neighborhood $W$ on the
unit sphere.  Taking the largest of the finitely many constants gives a
common entrance threshold and a common bound throughout $W$
(Corollary~\ref{cor:finite-cover}).  On unit directions outside $W$, the
minimum transformed score is bounded above by a fixed negative constant,
which permits the complementary decrease estimate.

Figure~\ref{fig:proof-overview} illustrates the construction.  Write $R$
for the target radius and $\varepsilon R$ for the inner threshold, where
$\varepsilon>0$ will be chosen sufficiently small.  Since $W$ consists
of unit vectors, states with direction in $W$ lie in the cone
$\mathcal C(W):=\{r\theta:r>0,\ \theta\in W\}$.
The shaded region is the part of this cone between the inner and outer
circles, not $W$ itself.
\begin{figure}[H]
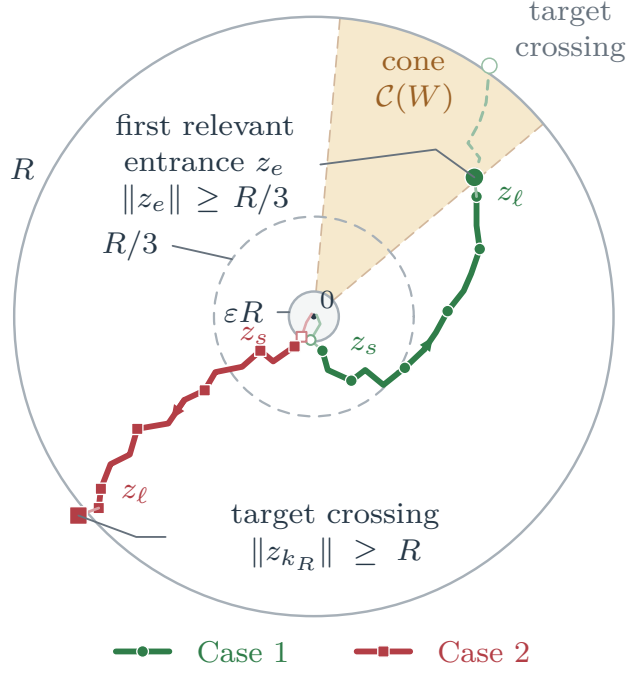

\centering
\resizebox{0.54\textwidth}{!}{\ProofOverviewCombined}
\captionsetup{width=.92\textwidth,justification=centering}
\caption{The two cases, with circles of radii $\varepsilon R$, $R/3$,
and $R$.  The highlighted segments run from $z_s$ to $z_\ell$.
Trajectories and step lengths are schematic; $\varepsilon=1/12$ is
used for the drawing and may be smaller in the proof.}
\label{fig:proof-overview}
\end{figure}

For trajectory segments whose unit directions remain outside $W$, we join
consecutive states by line segments and divide both the states and the
elapsed time by the target radius.  We choose these segments so that their
state norms lie between fixed positive multiples of the target radius.
The oracle errors divided by the target radius tend to zero.
Subsequential uniform limits of these
interpolated curves are the limit paths used in the proof.  Along these
limit paths, a nonnegative Lyapunov function decreases at a uniformly
positive rate.  We transfer this decrease to the original discrete
segments; we do not require decrease at every individual update
(Lemma~\ref{lem:block}).

To select a segment that yields a contradiction, we distinguish two
cases, shown in Figure~\ref{fig:proof-overview}.  In Case~1 (green), the
trajectory first enters $\mathcal C(W)$ at or above the inner threshold
at a state $z_e$, before reaching the target radius.  The entrance
estimate bounds all later norms in terms of $\norm{z_e}$ and the oracle
error.  Since the trajectory subsequently reaches the target radius,
the entrance norm must be at least a fixed positive fraction of that
radius; the proof gives $\norm{z_e}\geq R/3$.
In Case~2 (red), the trajectory reaches the target radius at
$z_{k_R}$ without an earlier such entrance.  In either case, we select
the highlighted segment from $z_s$ to $z_\ell$: it starts immediately
after the last state at or below the inner threshold and ends just
before the entrance or target crossing.  All its states have unit
directions outside $W$.
Bounded increments ensure that this segment contains at least a fixed
positive multiple of the target radius in updates.  We choose the inner
threshold as a sufficiently small fixed fraction of the target radius.
The decrease estimate then forces the Lyapunov function to decrease by
more than its value at the beginning of the segment, which is controlled
by the inner threshold.
This contradicts its nonnegativity.  Section~\ref{sec:proof} gives the
full proof.

Block and Levin \cite{BlockLevin1970} already used dimension induction,
local confinement, projection onto a hyperplane orthogonal to a fixed
direction, and a finite covering of the sphere.  We adapt this geometric
strategy to nonsymmetric scores.  Near a degenerate direction $d$, we
project onto $(Ad)^\perp$ along $d$; this projection need not be orthogonal.
The reduced trajectory satisfies the same approximate oracle condition,
with unchanged error $B$, and the reduced operator remains coercive.
Away from the degenerate directions, we analyze limit paths of normalized
trajectories and transfer their Lyapunov decrease to discrete trajectory
segments.  Together, these estimates replace the displacement bound for
proper chains used by Block and Levin.  Our methodological contribution
lies in this adaptation and combination for nonsymmetric coercive scores.

\subsection{Two algorithmic applications}\label{subsec:applications-intro}

We present two algorithmic consequences of Theorem~\ref{thm:main}.
The first is a last-iterate rate for Frank--Wolfe applied to affine
strongly monotone variational inequalities.  The second is finite
termination of a coordinate-correction algorithm for linear inequality
systems whose square coefficient matrix has positive-definite
symmetric part.

\paragraph{Frank--Wolfe for affine strongly monotone variational inequalities.}
Our first application is an $O(k^{-1})$
last-iterate norm bound for harmonic vertex-returning Frank--Wolfe applied
to an affine strongly monotone variational inequality on a polytope with
a relatively interior solution. 
An affine variational inequality is the problem of
finding $x\in K$ such that
\[
  \ip{\Phi(x)}{y-x}\geq0
  \qquad\text{for all }y\in K,
\]
where $K$ is a nonempty compact polytope and $\Phi(x)=Lx+a$ is an affine
operator.  We assume that $L$ is coercive on the feasible-direction
space $T_K:=\operatorname{lin}(K-K)$, as specified in
Theorem~\ref{thm:app-affine-fw}.  A \emph{linear minimization oracle} (LMO) on $K$
takes a vector $c$ as input and returns a point minimizing
$s\mapsto\ip{c}{s}$ over $K$.  An approximate LMO with additive error
$\delta\geq0$ returns a point $s\in K$ satisfying
\[
  \ip{c}{s}\leq\min_{v\in K}\ip{c}{v}+\delta.
\]
In our model, the oracle is required to return a vertex of
$K$ and is called with $c=\Phi(x_k)$ at iteration $k$.  An exact
vertex minimizer always exists because $K$ is a polytope; our results
allow arbitrary choices among admissible vertices.
Suppose that a solution $x^*$ lies in
$\operatorname{relint}K$.  If, for all integers $k\geq0$,
\[
  x_{k+1}=x_k+\frac{1}{k+1}(s_k-x_k),
  \qquad s_k\in\operatorname{vert}(K),
\]
where $s_k$ is the LMO output with additive error $\delta_k$, then the
transformed variables
\[
  z_k:=k(x_k-x^*),
  \qquad u_k:=s_k-x^*
\]
satisfy the additive recursion $z_{k+1}=z_k+u_k$.  The variational inequality
and the relative-interior assumption imply $\Phi(x^*)\perp T_K$.  Writing
$P_{T_K}$ for the orthogonal projection onto $T_K$ and
$A:=P_{T_K}L|_{T_K}$, the Frank--Wolfe oracle inequality therefore becomes,
for all integers $k\geq1$,
\[
  \ip{Az_k}{u_k}
  \leq \min_{u\in U}\ip{Az_k}{u}+k\delta_k,
  \qquad
  U:=\{v-x^*:v\in\operatorname{vert}(K)\},
\]
where $A$ is coercive but need not be symmetric and $\delta_k$ is the
original oracle error.  Since $0\in\conv U$, Theorem~\ref{thm:main}
applies from index $k=1$ whenever $\sup_{k\geq1}k\delta_k<\infty$.
It bounds $(z_k)_{k\geq1}$ and hence gives
$\norm{x_k-x^*}=O(k^{-1})$.  The first step satisfies $x_1=s_0\in K$,
so $\norm{z_1}\leq\operatorname{diam}(K)$ independently of the initial
point and the first oracle choice.

A central difficulty in obtaining a convergence rate is the behavior of
the linear minimization oracle on a polytope.  Arbitrarily small changes
in the linear objective can change its minimizing vertex, so successive
oracle outputs need not approach one another even when the iterates
converge.  Hough \cite[Section~2]{Hough2026} proved asymptotic convergence
for monotone variational inequalities on general compact convex sets,
including polytopes.  His rate estimates use additional geometric
regularity of the feasible set, and he identifies the jumps between
minimizing vertices as an obstacle to extending that analysis to
polytopes \cite[Sections~3.1--3.2]{Hough2026}.
Our robust cycling theorem addresses this rate question for affine
strongly monotone operators on polytopes with a relatively interior
solution.  Instead of estimating the change between successive oracle
outputs, we prove that the cumulative deviations $z_k=k(x_k-x^*)$
remain bounded.  This yields an $O(k^{-1})$ last-iterate norm bound
despite persistent switching between vertices, uniformly over all
admissible tie-breaking choices.  The argument requires neither
continuity of the vertex selection nor symmetry of the operator, and it
also allows oracle errors satisfying
$\sup_{k\geq1}k\delta_k<\infty$.
\begin{theorem}[Affine relative-interior Frank--Wolfe rate]
\label{thm:app-affine-fw}
Let $K$ be a nonempty compact polytope in a finite-dimensional Euclidean
space, let $T_K:=\operatorname{lin}(K-K)$, and let
$\Phi(x)=Lx+a$ satisfy
\[
 \ip{Lh}{h}\geq\mu\norm h^2\qquad\text{for all }h\in T_K
\]
for some $\mu>0$.  Suppose that the solution $x^*$ of
\[
 \ip{\Phi(x^*)}{s-x^*}\geq0\qquad\text{for all }s\in K
\]
belongs to $\operatorname{relint}K$.  Starting from $x_0\in K$, consider,
for all $k\in\mathbb N_0$,
\begin{align}\label{eq:affine-fw}
 x_{k+1}&=x_k+\frac1{k+1}(s_k-x_k),
 \qquad s_k\in\operatorname{vert}(K)\\
\label{eq:affine-fw-oracle}
 \ip{\Phi(x_k)}{s_k}
 &\leq\min_{s\in K}\ip{\Phi(x_k)}{s}+\delta_k,
\end{align}
where $(\delta_k)_{k\in\mathbb N_0}$ is assumed to be nonnegative.
If $\overline B:=\sup_{k\geq1}k\delta_k<\infty$, then there is a constant
$C=C(K,L,x^*)$, independent of $x_0$, the oracle choices, and the sequence
$(\delta_k)$, such that
\[
 \norm{x_k-x^*}\leq\frac{C(1+\overline B)}{k}
 \qquad\text{for all integers }k\geq1.
\]
For an exact oracle, one may take $\overline B=0$.
\end{theorem}
For quadratic minimization, the equivalence with herding and the
$O(k^{-1})$ norm rate at a relative-interior solution are already established
by Bach, Lacoste-Julien, and Obozinski
\cite[Sections~4 and~4.2]{BachEtAl2012}.   Theorem~\ref{thm:app-affine-fw} extends this
guarantee to general affine strongly monotone operators whose compression
to the feasible-direction space $T_K:=\operatorname{lin}(K-K)$ need not
be symmetric.  In particular, it imposes no extra
small-coupling condition on the skew-symmetric part, unlike the interior
saddle-point analysis of Gidel, Jebara, and Lacoste-Julien
\cite[Theorem~1(I)]{GidelEtAl2017}.  We retain the
direct harmonic recursion with one linear minimization oracle per
iteration.
The proof of Theorem~\ref{thm:app-affine-fw} and the reciprocal-step
extension are given in Subsection~\ref{subsec:affine-fw-app}.
A short \hyperref[ex:saddle]{quadratic saddle-point example} in that
subsection illustrates why nonsymmetry matters: its bilinear interaction
is modeled by the skew-symmetric part of the operator.

\paragraph{Oblique relaxation and finite coordinate correction.}

A second application concerns systems of linear inequalities and
\emph{correction} algorithms that seek a feasible point through simple
updates, without computing a projection or a matrix inverse.
Agmon \cite{Agmon1954} developed a classical relaxation method for
this feasibility problem, using corrections along constraint normals.
Amaldi and Hauser \cite{AmaldiHauser2005} later quantified boundedness
for variants of such corrections, even when the inequality system is
inconsistent.  Our application allows the correction directions to be
transformed by a nonsymmetric coercive matrix.

For $G\in\R^{n\times n}$ with positive-definite symmetric
part, consider $Gx\geq b$.  Such a matrix is invertible, so the system
is always feasible.  The question is whether the following restricted
correction rule finds a feasible point in finitely many steps.
Whenever the system is violated, increase
the coordinate indexed by a most-violated inequality by one.  This
correction can worsen other inequalities, and its length does not
decrease as the residual becomes small.  Nevertheless, the cycling
theorem implies finite termination.  The reduction uses
$z_k=x_k-G^{-1}b$ and the update set $\{0,e_1,\ldots,e_n\}$, with zero
updates after termination.  Zero is a vertex of this update hull, so
the application uses the boundary case of Theorem~\ref{thm:main}.
Boundedness
forces termination because each nonzero update increases the sum of
the coordinates by one.
Proposition~\ref{prop:oblique-relaxation} proves boundedness for greedy
corrections through a fixed coercive matrix, without assuming
feasibility.  Its concrete finite-termination consequence is the
following coordinate rule, where $e_i$ denotes the $i$th standard basis
vector.  An additive error in selecting the index is allowed, but the
stopping test checks the true residuals.

\begin{corollary}[Finite coordinate correction]
\label{cor:coordinate-termination}
Let $n\geq1$, let $G\in\R^{n\times n}$ have positive-definite symmetric
part, and let $b\in\R^n$ and $B\geq0$.  Starting from any $x_0\in\R^n$,
stop at index $k$ if $Gx_k\geq b$ componentwise.  Otherwise choose
$i_k\in\{1,\ldots,n\}$ and update according to
\begin{equation}\label{eq:coordinate-correction}
 (Gx_k-b)_{i_k}\leq\min_{1\leq i\leq n}(Gx_k-b)_i+B,
 \qquad x_{k+1}=x_k+e_{i_k}.
\end{equation}
For every admissible choice of indices, the procedure stops after a
finite number $T$ of corrections at a point satisfying $Gx_T\geq b$.
There is a constant $C=C(n,G)\geq1$, independent of $b$, $B$, $x_0$,
and the selections, such that, with $\rho=\norm{x_0-G^{-1}b}$,
\begin{equation}\label{eq:coordinate-count}
 T\leq\sqrt n\bigl[C(1+B)+(C+1)\rho\bigr].
\end{equation}
\end{corollary}

For $B=0$ the rule always selects a most-violated inequality.
For $B>0$ even a nonviolated inequality can satisfy the selection
condition; finite termination still follows from the same bound.
Amaldi and Hauser \cite[Theorem~7.1]{AmaldiHauser2005} already covered
the symmetric case after a change of coordinates.  Frommer and Szyld
\cite[Section~4]{FrommerSzyld2023} analyzed greedy coordinate relaxation
for linear equalities with nonsymmetric coefficient matrices under
generalized diagonal-dominance assumptions, using residual-dependent
correction lengths.
Here we obtain finite termination for inequalities with fixed unit
corrections when the symmetric part of $G$ is positive definite,
without an additional diagonal-dominance condition or restriction on
the skew part.
Subsection~\ref{subsec:oblique-relaxation} gives the proof, the general
preconditioned formulation, and the extension to bounded positive
step lengths.  The guarantee concerns the specified correction rule,
not a general feasibility test or a runtime comparison with other solvers.
Section~\ref{sec:limitations} discusses the scope of these
consequences and possible extensions.

\subsection{Organization}

Section~\ref{sec:related} reviews perceptron cycling, affine
Frank--Wolfe, and relaxation methods for
linear systems.
Section~\ref{sec:proof} contains the proof of the main theorem, the
rounding proof for weighted updates, the explicit bound for symmetric
linear operators, and a counterexample for settings without coercivity.
Section~\ref{sec:applications} develops the two algorithmic applications:
affine Frank--Wolfe, illustrated by a quadratic saddle problem, and
oblique relaxation with finite coordinate correction.
Section~\ref{sec:limitations} describes further consequences and
generalizations, and discusses limitations and open problems.
Appendix~\ref{app:oracle-errors} gives the detailed reduction of
imperfect oracle information to an additive score error.

\section{Related work}\label{sec:related}

\paragraph{Perceptron convergence and cycling.}
Rosenblatt \cite{Rosenblatt1958} introduced the perceptron.  Agmon
\cite{Agmon1954} had earlier studied a relaxation method for consistent
systems of linear inequalities, to which the perceptron's
finite-termination theory for separable data is connected.  Novikoff
\cite{Novikoff1962} subsequently gave a short convergence proof for
perceptrons.  Shalev-Shwartz and Ben-David
\cite[Theorem~9.1]{ShalevShwartzBenDavid2014} provided a modern account.
These convergence results concern the separable case, rather than
cycling for nonseparable data.

Block and Levin \cite{BlockLevin1970} reported that Nils Nilsson and,
independently, Terry Beyer had conjectured boundedness in the nonseparable
case.  Efron \cite{Efron1964} offered an early proof of boundedness for
the perceptron correction procedure in the nonseparable case in a 1964
technical report, as reported by Block and Levin \cite{BlockLevin1970}.
Minsky and Papert \cite[Chapter~11]{MinskyPapert1969} subsequently
presented an argument for the same boundedness theorem.  Block and Levin
\cite{BlockLevin1970} identified a gap in Minsky and Papert's argument
and supplied a rigorous proof of the boundedness statement
\eqref{eq:classical-bound-intro} using a stronger lemma that bounds the
displacement along proper correction chains.  Amaldi and Hauser
\cite[Theorem~7.1]{AmaldiHauser2005} later established broader
relaxation-method boundedness theorems with explicit bounds in terms of
the constraint data and an upper bound on the step coefficients.
In fact, their Theorem~7.1 implies
Theorem~\ref{thm:main} when $A$ is symmetric positive definite, including
positive oracle errors and $0$ on the boundary of $\conv U$.
To see this, set $V:=A^{1/2}U$ and
\[
 y_k=A^{1/2}z_k,\qquad v_k=A^{1/2}u_k,\qquad
 \rho=\max\{1,B,\norm{y_0}\},\qquad q_k=y_k/\rho.
\]
Since $0\in\conv U$, the minimum score is nonpositive, and hence
\[
 q_{k+1}=q_k+\rho^{-1}v_k,\qquad
 \ip{q_k}{v_k}\leq B/\rho\leq1,\qquad \norm{q_0}\leq1.
\]
The case $U=\{0\}$ is immediate.  Otherwise, omit zero updates, which
do not change the states, and apply their theorem with the nonzero
vectors in $V$ as constraint normals and with all right-hand sides
equal to~$1$.  Their hypothesis allows the non-strict inequality
$\ip{q_k}{v_k}\leq1$, including equality at the threshold.
Their step coefficients multiply the unnormalized normals: here they
are $\rho^{-1}\in(0,1]$.  The theorem requires no positive lower bound
and provides one bound for all coefficient sequences in $[0,1]$.
The actual displacement lengths are $\rho^{-1}\norm{v_k}$, not
necessarily at most~$1$.  The theorem imposes no feasibility assumption; indeed,
the full system $\ip{v}{q}>1$ for all $v\in V$ is infeasible because
$0\in\conv V$.  Removing zero normals does not affect applicability,
whether or not the remaining system is feasible.  The proof of their
estimate applies to finite prefixes as well, so it also covers the case
of only finitely many nonzero updates.
Their explicit inequality~(7.1) depends on the fixed constraint data,
the upper bound~$1$ on the step coefficients, and
$\norm{(q_0,1)}\leq\sqrt2$; hence it bounds all $q_k$ by a constant
$M(E,V)$ independent of $\rho$.  Multiplying by $\rho$ and applying
$A^{-1/2}$ gives
$\sup_k\norm{z_k}\leq C(E,U,A)(1+\norm{z_0}+B)$.
Their bound uses a condition number of a homogenized constraint matrix;
Theorem~\ref{thm:explicit-symmetric} gives a simpler expression in terms
of the relative inradius under the additional relative-interior assumption.
This reduction, however, does not cover general nonsymmetric scores.  A common linear
change of coordinates $z\mapsto Tz$ and $u\mapsto Tu$ produces the score
$\ip{Tz}{Tu}=\ip{T^*Tz}{u}$, whose matrix is symmetric; here $^*$ denotes
the adjoint.  Moreover, writing
$A=S+J$ with $S^*=S$ and $J^*=-J$, the additional term $\ip{Jz}{u}$ is
generally unbounded as $z$ varies.  It cannot be treated as a fixed additive
error in the symmetric theorem.  Theorem~\ref{thm:main} establishes
boundedness directly for the nonsymmetric score, assuming only coercivity
and approximate global minimization.  Corollary~\ref{cor:weighted} also
allows bounded nonnegative step coefficients in this nonsymmetric setting.

Gelfand, Chen, van der Maaten, and Welling \cite{GelfandEtAl2010}
connected the classical $O(T^{-1})$ averaging consequence to herding and
its variants.  Their sign condition $\ip{w_t}{v_t}\leq0$
allows updates that need not be approximate global minimizers;
Theorem~\ref{thm:main} therefore does not subsume every such update.
Harvey and Samadi \cite{HarveySamadi2014} improved the dimension and
cardinality dependence using a different sampling algorithm.

\paragraph{Approximate and delayed selection.}
Bertsekas and Tsitsiklis \cite{BertsekasTsitsiklis1989} systematically
studied convergence of asynchronous iterations in which components are
updated using delayed information, including bounded-delay assumptions
and their interaction with contraction properties.  Gray and Neuhoff
\cite{GrayNeuhoff1998} surveyed the theory of quantization in information
processing.  Freund and Grigas
\cite[Section~5.2, Proposition~5.1]{FreundGrigas2016} analyzed inexact
gradients and linear minimization subproblems for the Frank--Wolfe method.
They compared the true and approximate linear scores at both the selected
and optimal points and also treated simultaneous gradient and subproblem
errors.

Proposition~\ref{prop:robustness-mechanisms} in
Appendix~\ref{app:oracle-errors} uses this standard comparison
and the elementary estimate
$\norm{z_k-z_{k-\tau_k}}\leq\tau\max_{u\in U}\norm u$ to give a single tolerance for all
the stated imperfections.  The error calculation itself needs neither
coercivity nor $0\in\conv U$; those hypotheses enter when
Theorem~\ref{thm:main} is applied.  Combining this reduction with the
theorem gives a uniform state bound in the nonsymmetric setting despite persistent bounded
errors, together with an $O(T^{-1})$ average-update residual by telescoping.

\paragraph{Affine Frank--Wolfe for variational inequalities.}
Frank and Wolfe \cite{FrankWolfe1956} introduced the method for smooth
convex optimization.  Jaggi \cite{Jaggi2013} provided a modern account,
and Freund and Grigas \cite{FreundGrigas2016} treated inexact gradients
and linear oracles.

Bach, Lacoste-Julien, and Obozinski \cite[Sections~4 and~4.2]{BachEtAl2012}
identified quadratic Frank--Wolfe with herding.  For harmonic steps and a
relative-interior target, their analysis yields $O(k^{-2})$ objective
error and thus $O(k^{-1})$ norm error.  These rates cover our symmetric
affine VI case after a change of coordinates.  Lacoste-Julien, Lindsten,
and Bach \cite[Supplement, Theorem~G.1]{LacosteJulienEtAl2015} obtained
an $O(k^{-1})$ norm term plus an $O(\delta)$ floor for approximate herding
with fixed original-oracle error $\delta$; our fixed additive-state
tolerance $B$ instead corresponds to $\delta_k=B/k$.
Wirth, Kerdreux, and Pokutta \cite[Theorem~3.6]{WirthEtAl2023} obtained
the same objective and norm orders for smooth strongly convex minimization
with an interior solution using steps $4/(k+4)$.

Hammond \cite{Hammond1984} studied direct conditional-gradient recursions
for asymmetric VIs.  Gidel, Jebara, and Lacoste-Julien
\cite[Theorem~1(I)]{GidelEtAl2017} established geometric convergence of
the saddle error for ordinary interior saddle-point Frank--Wolfe with
adaptive steps, under a condition relating coupling, curvature, and
geometry.  Their separate polytope result uses away steps; their
open-loop result uses different steps and controls a best-iterate gap.

Chen and Mazumdar \cite{ChenMazumdar2024} studied smoothed generalized
Frank--Wolfe, including a stochastic version.  Rahimi Baghbadorani,
Mohajerin Esfahani, and Grammatico \cite{RahimiEtAl2026} used
Frank--Wolfe to solve inner saddle problems within an accelerated method.
Neither approach uses the direct recursion \eqref{eq:affine-fw}.
Hough \cite[Section~2]{Hough2026} proved convergence of that recursion for
monotone $C^1$ operators on compact convex sets, including polytopes,
thereby resolving Hammond's conjecture.  His gap-rate estimates require
strong convexity of the feasible set.  That work left the general
polytope rate question open.

Theorem~\ref{thm:app-affine-fw} gives an $O(k^{-1})$ last-iterate norm
bound for the affine, relatively interior polytope case, with harmonic
steps and no additional restriction on the skew part.  Its constant
depends on the fixed problem data and the finite tolerance bound
$\sup_{k\geq1}k\delta_k$, not on initial points or admissible
vertex-oracle selections.
Proposition~\ref{prop:affine-fw-schedules} extends this result to general steps of the form $\alpha/(k+\beta)$ and polynomially
decaying oracle errors.

\paragraph{Oblique relaxation and coordinate corrections.}
Amaldi and Hauser \cite[Theorem~7.1]{AmaldiHauser2005} proved boundedness
for corrections along constraint normals, including inconsistent systems
and bounded nonnegative step coefficients.  For a symmetric
positive-definite preconditioner $P$, the change of variables
$y=P^{-1/2}x$ puts corrections $x\mapsto x+Pa_i$ into their framework.
Proposition~\ref{prop:oblique-relaxation} allows nonsymmetric coercive
$P$, but requires approximate minimization of the residual rather than
arbitrary selection of a violated constraint.  Thus, neither result implies
the other in general.

Frommer and Szyld \cite[Sections~2 and~4]{FrommerSzyld2023} analyzed
randomized and greedy coordinate relaxation for linear equalities.
Their nonsymmetric Gauss--Southwell results used $H$-matrix conditions,
equivalently generalized diagonal dominance, and residual-dependent
correction lengths.  Corollary~\ref{cor:coordinate-termination} concerns
inequalities, fixed positive unit corrections, and positive definiteness
of the symmetric part.  It establishes finite termination for this rule
without symmetry or an additional diagonal-dominance assumption.
It does not assert convergence to the equality solution or a
computational advantage over their methods.

Guzm\'an and Klivans \cite[Sections~2--3]{GuzmanKlivans2016} extended
chip-firing from $M$-matrices to arbitrary invertible integer matrices
by changing the cone of admissible configurations and allowing a firing
only when the new configuration remains in that cone.  Our residual
update $r\mapsto r-Ge_i$, with $r=b-Gx$, also subtracts a matrix column.
However, the selection uses a largest positive residual, and the
stopping condition is $r\leq0$ in the original coordinates.  Their
admissibility condition is different and does not directly imply this
finite-termination result.

\section{Proof of the robust cycling theorem}\label{sec:proof}

The proof proceeds by induction on $n=\dim E$.  The case $n=0$ is immediate.
If $U=\{0\}$, every trajectory is constant and
Theorem~\ref{thm:main} holds with $C=1$.  Henceforth assume $n\geq1$ and
$U\neq\{0\}$, and assume the theorem in every ambient dimension
$\ell<n$.

If the theorem failed in dimension $n$, there would be approximate
trajectories reaching target radii arbitrarily large relative to
$1+\norm{z_0}+B$.  After the preliminaries, we establish two estimates.
Subsection~\ref{subsec:case-degenerate} uses face reduction and the
induction hypothesis to bound all subsequent norms after entrance into
a fixed neighborhood of the degenerate directions above a prescribed
radial threshold.  Subsection~\ref{subsec:case-nondegenerate} uses limit
paths and the Lyapunov function $H$ to obtain a uniform decrease estimate
for segments whose unit directions remain outside that neighborhood.

Subsection~\ref{subsec:completion} constructs the counterexample sequence
and distinguishes two cases.  If such an entrance precedes the target
crossing, the entrance estimate shows that the state at first entrance
has norm at least a fixed positive fraction of the target radius.
Otherwise, the target crossing itself supplies the endpoint.
In either case, a preceding segment outside the
neighborhood has a number of updates proportional to the target radius.
The required decrease of $H$ exceeds its initial value on this segment,
contradicting nonnegativity.  Figure~\ref{fig:proof-overview} illustrates
the segment selected in each case.

\subsection{Preliminaries}

Let $R_U:=\max_{u\in U}\norm u$ denote the maximum norm of an
admissible increment.  Since $U\neq\{0\}$, we have $R_U>0$.
Write $\Sph:=\{\theta\in E:\norm{\theta}=1\}$.
For $z\in E$, define
\begin{align*}
  H(z)&:=-\min_{u\in U}\ip{Az}{u},\qquad
  \mathcal F(z):=\argmin_{u\in U}\ip{Az}{u},\\
  N_0&:=\{\theta\in\Sph:H(\theta)=0\}.
\end{align*}
The following claims record the properties used repeatedly below.

\begin{claim}[properties of the Lyapunov function]
\label{clm:H-properties}
The function $H$ is nonnegative, convex, and positively homogeneous.  For
$z\neq0$,
\[
  H(z)=\norm{z}\,H\left(\frac z{\norm z}\right).
\]
Moreover, with $b:=\norm A R_U$,
\[
  |H(z)-H(z')|\leq b\norm{z-z'}
  \qquad\text{for all }z,z'\in E.
\]
\end{claim}

\begin{proof}
Since $0\in\conv U$, choose coefficients $\lambda_u\geq0$ with
$\sum_{u\in U}\lambda_u=1$ and $\sum_{u\in U}\lambda_u u=0$.  Then
\[
  \min_{u\in U}\ip{Az}{u}
  \leq\sum_{u\in U}\lambda_u\ip{Az}{u}=0,
\]
so $H(z)\geq0$.  The representation
\[
  H(z)=\max_{u\in U}\ip{-Az}{u}
\]
shows that $H$ is the maximum of finitely many linear functions and hence is
convex.  For every $\lambda\geq0$, linearity gives
$H(\lambda z)=\lambda H(z)$, proving positive homogeneity and the displayed
radial identity.

Finally, for any finite family of real numbers $(a_u)$ and $(b_u)$,
$|\max_u a_u-\max_u b_u|\leq\max_u|a_u-b_u|$.  Therefore
\begin{align*}
 |H(z)-H(z')|
 &\leq\max_{u\in U}|\ip{A(z-z')}{u}|\\
 &\leq\norm A R_U\norm{z-z'},
\end{align*}
where the last step is the Cauchy--Schwarz inequality.
\end{proof}

\begin{samepage}
\begin{claim}[compactness of the degenerate directions]
\label{clm:N0-compact}
The set $N_0$ is compact.
\end{claim}

\begin{proof}
Claim~\ref{clm:H-properties} shows that $H$ is Lipschitz and therefore
continuous.  Hence $N_0=\Sph\cap H^{-1}(\{0\})$ is closed.  The unit sphere
$\Sph$ is compact by the finite-dimensional Heine--Borel theorem in
Rudin \cite[Theorem~2.41]{RudinPMA}; consequently its closed subset $N_0$ is
compact.
\end{proof}
\end{samepage}

\begin{claim}[local stability of the minimizing face]
\label{clm:face-stability}
For $x\in E$, put
\[
  g(x):=\min\left\{
  \ip{Ax}{u}-\min_{v\in U}\ip{Ax}{v}:u\in U\setminus\mathcal F(x)
  \right\},
\]
with $g(x)=+\infty$ if $\mathcal F(x)=U$.  Suppose that
$\mathcal F(x)\neq U$ and
\[
  \norm{x'-x}\leq\frac{g(x)}{4\norm A R_U},
  \qquad 0\leq\beta<\frac{g(x)}2.
\]
Then every $\beta$-optimal element at $x'$, meaning every $u\in U$ with
\[
  \ip{Ax'}{u}\leq\min_{v\in U}\ip{Ax'}{v}+\beta,
\]
belongs to $\mathcal F(x)$.  If $\mathcal F(x)=U$, the conclusion is
automatic.
\end{claim}

\begin{proof}
Assume $\mathcal F(x)\neq U$.  For every $u\in U\setminus\mathcal F(x)$,
\[
  \ip{Ax}{u}-\min_{v\in U}\ip{Ax}{v}>0.
\]
Because $U\setminus\mathcal F(x)$ is finite and nonempty, the minimum of
these positive numbers is positive; thus $g(x)>0$.  Put
$b:=\norm A R_U$.  For every $v\in U$, the Cauchy--Schwarz inequality gives
\[
 \bigl|\ip{Ax'}{v}-\ip{Ax}{v}\bigr|
 =\bigl|\ip{A(x'-x)}{v}\bigr|
 \leq\norm A\,\norm{x'-x}\,\norm v
 \leq b\norm{x'-x}.
\]
Thus changing the state from $x$ to $x'$ changes each score by at most
$b\norm{x'-x}$.  Fix $u\in U\setminus\mathcal F(x)$.  The preceding
estimate gives the lower bound
\[
 \ip{Ax'}{u}\geq\ip{Ax}{u}-b\norm{x'-x}.
\]
To bound the minimum score from above, choose a minimizer
$v_*\in\mathcal F(x)$, which exists because $U$ is finite and nonempty.
Then
\begin{align*}
 \min_{v\in U}\ip{Ax'}{v}
 &\leq\ip{Ax'}{v_*}\\
 &=\ip{Ax}{v_*}+\ip{A(x'-x)}{v_*}\\
 &\leq\ip{Ax}{v_*}+b\norm{x'-x}\\
 &=\min_{v\in U}\ip{Ax}{v}+b\norm{x'-x}.
\end{align*}
The second inequality follows from
$\ip{A(x'-x)}{v_*}\leq\norm A\,\norm{x'-x}\,\norm{v_*}
\leq b\norm{x'-x}$, since $\norm{v_*}\leq R_U$ and $b=\norm A R_U$.
The element $v_*$ need not remain a minimizer at $x'$: its score merely
provides an upper bound on the minimum at $x'$.  Subtracting this upper
bound from the lower bound on $\ip{Ax'}{u}$, and then using the definition
of $g(x)$ and the assumed bound on $\norm{x'-x}$, yields
\begin{align*}
 \ip{Ax'}{u}-\min_{v\in U}\ip{Ax'}{v}
 &\geq \ip{Ax}{u}-b\norm{x'-x}
      -\left(\min_{v\in U}\ip{Ax}{v}+b\norm{x'-x}\right)\\
 &\geq g(x)-2b\norm{x'-x}\\
 &\geq\frac{g(x)}2>\beta.
\end{align*}
Thus no element outside $\mathcal F(x)$ is $\beta$-optimal at $x'$.  If
$\mathcal F(x)=U$, every element of $U$ already belongs to $\mathcal F(x)$,
so the conclusion is immediate.
\end{proof}

\begin{claim}[directional minimizer bound]
\label{clm:directional-bound}
If $z\neq0$, $\theta=z/\norm z$, and
$w\in\conv\mathcal F(z)$, then
\[
  \ip{A\theta}{w}=-H(\theta),
  \qquad
  \norm w\geq\frac{H(\theta)}{\norm A}.
\]
\end{claim}

\begin{proof}
Positive rescaling of the objective does not change its minimizers, so
$\mathcal F(z)=\mathcal F(\theta)$.  Every $u\in\mathcal F(\theta)$ satisfies
\[
  \ip{A\theta}{u}
  =\min_{v\in U}\ip{A\theta}{v}=-H(\theta).
\]
The same equality holds for every convex combination
$w\in\conv\mathcal F(\theta)$.  Since $H(\theta)\geq0$, the
Cauchy--Schwarz inequality yields
\[
  H(\theta)=-\ip{A\theta}{w}
  \leq\norm{A\theta}\norm w
  \leq\norm A\norm w,
\]
which proves the norm bound.
\end{proof}

\subsection{Control near degenerate directions}
\label{subsec:case-degenerate}
We now analyze $B$-approximate trajectories as defined in
Definition~\ref{def:B-trajectory}.
This subsection first establishes an entrance estimate.  We show that if an
iterate has norm above an explicit entrance threshold and its unit direction belongs to a prescribed
neighborhood of $N_0$, then every subsequent state is bounded in terms of
the entrance norm.  In the completion of the induction, this estimate
will show that any such entrance before the target radius is reached must
occur at a norm at least one third of that radius.
For each $d\in N_0$, we first identify
the face $F_d$ of updates that preserve the coordinate normal to that face.
Lemma~\ref{lem:reduction} shows that, while the selected updates belong to
$F_d$, the remaining component is a lower-dimensional approximate
trajectory.  Then
Lemma~\ref{lem:trapping-condition} gives a pointwise condition that forces
every admissible update to belong to $F_d$.  Both lemmata are used in
Proposition~\ref{prop:trapping}, together with the inductive hypothesis, to
show that explicit conditions on the norm and unit direction of one state
make this pointwise condition valid at every subsequent iterate.  Finally,
Corollary~\ref{cor:finite-cover} replaces the constants associated with an
individual $d$ by constants valid on one neighborhood of the compact set
$N_0$.

Fix $d\in N_0$ and set
\[
  F_d:=\mathcal F(d)=\{u\in U:\ip{Ad}{u}=0\},
  \qquad H_d:=(Ad)^\perp.
\]
If $F_d\neq U$, define
\[
  \delta_d:=\min_{u\in U\setminus F_d}\ip{Ad}{u}>0.
\]
Since $\ip{Ad}{d}\geq\mu>0$, the subspace $H_d=(Ad)^\perp$ has
codimension one and does not contain $d$, so
$E=\operatorname{span}\{d\}\oplus H_d$.  From this unique decomposition we
get, for every $z\in E$,
\begin{equation}\label{eq:decomposition}
  z=\alpha(z)d+y(z),\qquad
  \alpha(z):=\frac{\ip{Ad}{z}}{\ip{Ad}{d}},\qquad y(z)\in H_d.
\end{equation}
Indeed, the requirement $z-\alpha(z)d\in(Ad)^\perp$ gives
\[
  0=\ip{Ad}{z}-\alpha(z)\ip{Ad}{d},
\]
which determines $\alpha(z)$ because $\ip{Ad}{d}>0$.

We next verify that $0\in\conv F_d$.  Since $d\in N_0$, we have
\[
  0=H(d)=-\min_{u\in U}\ip{Ad}{u},
\]
and hence $\ip{Ad}{u}\geq0$ for every $u\in U$.  Because $0\in\conv U$,
there are coefficients $\lambda_u\geq0$ such that
\[
  \sum_{u\in U}\lambda_u=1,
  \qquad
  \sum_{u\in U}\lambda_u u=0.
\]
Taking the inner product with $Ad$ gives
\[
  0=\sum_{u\in U}\lambda_u\ip{Ad}{u}.
\]
Every term in this sum is nonnegative.  Therefore $\lambda_u>0$ implies
$\ip{Ad}{u}=0$, or equivalently $u\in F_d$.  The same coefficients thus
express zero as a convex combination of elements of $F_d$.

Let $P_{H_d}$ be the orthogonal projection onto $H_d$, and define the
reduced operator
\[
  A_d:=P_{H_d}A|_{H_d}\colon H_d\to H_d.
\]
Then $F_d\subset H_d$, $0\in\conv F_d$, and
\[
  \ip{A_dy}{y}=\ip{Ay}{y}\geq\mu\norm y^2
  \qquad\text{for all }y\in H_d.
\]
The reduced space $H_d$ has dimension $n-1$.  The
induction hypothesis, applied with $\ell=n-1$ to $(H_d,F_d,A_d)$, therefore
provides the constant $C(H_d,F_d,A_d)$ from Theorem~\ref{thm:main}.  For
brevity, set
\[
  C_d:=C(H_d,F_d,A_d).
\]
The subscript records that the reduced triple depends on $d$; for fixed $d$,
$C_d$ is independent of $B$, the initial state, and the admissible reduced
trajectory.  By possibly increasing $C_d$, we may assume without loss of
generality that
\begin{equation}\label{eq:Cd-at-least-one}
  C_d\geq1.
\end{equation}
This normalization is used explicitly in the initial trapping estimate.
The following lemma describes the dynamics conditional on selecting updates from
$F_d$.  The second lemma gives the explicit inequality under which the
$B$-approximate oracle cannot select an update outside $F_d$.

\begin{lemma}[reduction]\label{lem:reduction}
Suppose $(z_k)$ is $B$-approximate and $u_k\in F_d$ for
$k_0\leq k<k_1$.  Then $\alpha(z_k)=\alpha(z_{k_0})$ for
$k_0\leq k\leq k_1$, and $(y(z_k))_{k_0\leq k\leq k_1}$ is a
$B$-approximate trajectory segment for $(H_d,F_d,A_d)$.
\end{lemma}

\begin{proof}
For $u\in F_d$, \eqref{eq:decomposition} gives $\alpha(u)=0$ and $y(u)=u$.
Thus the normal coordinate is constant and
$y(z_{k+1})=y(z_k)+u_k$.  Moreover, for $u\in F_d$,
\[
  \ip{Az_k}{u}=\ip{A_dy(z_k)}{u}.
\]
Using $F_d\subseteq U$ in \eqref{eq:trajectory} now gives the required
$B$-approximate inequality over $F_d$.
\end{proof}

\begin{lemma}[trapping condition]\label{lem:trapping-condition}
Suppose $F_d\neq U$ and let $z\in E$.  If
\begin{equation}\label{eq:trapping-condition}
  \alpha(z)\,\delta_d
  >2\norm A R_U\norm{y(z)}+B,
\end{equation}
then every $B$-optimal element of $U$ at $z$ belongs to $F_d$.
If $F_d=U$, the conclusion is vacuous and holds without a condition.
\end{lemma}

\begin{proof}
Condition~\eqref{eq:trapping-condition} implies $\alpha(z)>0$.  Hence, for
$w\in U\setminus F_d$, the definition of $\delta_d$ gives
\[
  \ip{Az}{w}
  \geq\alpha(z)\delta_d-\norm A R_U\norm{y(z)},
\]
whereas
\[
  \min_{u\in U}\ip{Az}{u}
  \leq\min_{u\in F_d}\ip{Ay(z)}{u}
  \leq\norm A R_U\norm{y(z)}.
\]
Condition \eqref{eq:trapping-condition} makes the first quantity strictly
larger than the second plus $B$, so $w$ cannot be $B$-optimal.
\end{proof}

The proposition below closes the resulting feedback argument.  Its
assumptions on $z_{k_0}$ imply \eqref{eq:trapping-condition} at time $k_0$.  The
lower-dimensional bound supplied by Lemma~\ref{lem:reduction} and the
inductive hypothesis then keeps $y(z_k)$ below the threshold in
\eqref{eq:trapping-condition}, so the selected updates remain in $F_d$ for
all $k\geq k_0$.

\begin{proposition}[permanent trapping]\label{prop:trapping}
If $F_d\neq U$, put
\begin{equation}\label{eq:eta-K-proper}
  \eta_d:=\frac{\delta_d}{4\norm A R_U C_d},
  \qquad
  K_d:=\frac{2(2\norm A R_U C_d+1)}{\delta_d}.
\end{equation}
If $F_d=U$, put instead
\begin{equation}\label{eq:eta-K-full}
  \eta_d:=\frac1{4C_d},
  \qquad K_d:=0.
\end{equation}
Choose $\varepsilon_d>0$ small enough such that
\begin{equation}\label{eq:cone}
  \frac12\leq\alpha(\theta)\leq\frac32,
  \qquad
  \norm{y(\theta)}\leq\eta_d\alpha(\theta)
\end{equation}
for all $\theta\in\Sph$ with $\norm{\theta-d}<\varepsilon_d$.
If a $B$-approximate trajectory satisfies, at some time $k_0$,
\[
  \rho:=\norm{z_{k_0}}>2K_d(1+B),
  \qquad
  \left\lVert\frac{z_{k_0}}\rho-d\right\rVert<\varepsilon_d,
\]
then $u_k\in F_d$ and $\alpha(z_k)=\alpha(z_{k_0})$ for every $k\geq k_0$,
and
\begin{equation}\label{eq:trapping-bound}
  \norm{z_k}\leq2\rho+C_d(1+B)
  \qquad\text{for all integers }k\geq k_0.
\end{equation}
\end{proposition}

\begin{proof}
Write $\alpha_0:=\alpha(z_{k_0})$ and $y_0:=y(z_{k_0})$.  From
\eqref{eq:cone},
\begin{equation}\label{eq:alpha-y-initial}
  \frac\rho2\leq\alpha_0\leq\frac{3\rho}2,
  \qquad \norm{y_0}\leq\eta_d\alpha_0.
\end{equation}

First suppose $F_d\neq U$.  By \eqref{eq:alpha-y-initial} and the assumption
$\rho>2K_d(1+B)$, we have $\alpha_0\geq\rho/2>K_d(1+B)$.
We now verify \eqref{eq:trapping-condition} at $z_{k_0}$ by the following
chain of inequalities:
\begin{align*}
  2\norm A R_U\norm{y_0}+B
  &\leq 2\norm A R_U\eta_d\alpha_0+B\\
  &=\frac{\delta_d}{2C_d}\alpha_0+B\\
  &\leq\frac{\delta_d}{2}\alpha_0+B
   <\delta_d\alpha_0.
\end{align*}
Here the penultimate inequality uses $C_d\geq1$, and the last one follows
from $\alpha_0>K_d(1+B)$ and
$K_d\geq2/\delta_d$.

Suppose that \eqref{eq:trapping-condition} fails at some state $z_k$ with
$k\geq k_0$, and define the first such index by
\[
  k_1:=\min\left\{k\in\mathbb N_0:k\geq k_0,\
    \alpha(z_k)\delta_d\leq2\norm A R_U\norm{y(z_k)}+B\right\}.
\]
The minimum exists by this supposition, and $k_1>k_0$ because the strict
inequality holds at $k_0$.  For every $k_0\leq k<k_1$,
\eqref{eq:trapping-condition} holds, so Lemma~\ref{lem:trapping-condition}
gives $u_k\in F_d$.  Lemma~\ref{lem:reduction} therefore gives a reduced
$B$-approximate trajectory segment $(y(z_k))_{k_0\leq k\leq k_1}$.
Although this segment is finite, the inductive hypothesis applies because it
extends to an infinite $B$-approximate trajectory: at every subsequent state
one may choose an exact minimizer from the nonempty finite set $F_d$, and an
exact minimizer is in particular $B$-optimal.  Hence the inductive bound
applies to the segment and gives
\[
  \alpha(z_{k_1})=\alpha_0,
  \qquad
  \norm{y(z_{k_1})}
  \leq C_d(1+\norm{y_0}+B)
  \leq C_d(1+\eta_d\alpha_0+B).
\]
Consequently
\begin{align*}
  2\norm A R_U\norm{y(z_{k_1})}+B
  &\leq\frac{\delta_d}{2}\alpha_0
       +(2\norm A R_U C_d+1)(1+B)\\
  &<\frac{\delta_d}{2}\alpha_0
       +\frac{\delta_d}{2}\alpha_0
   =\delta_d\alpha_0.
\end{align*}
Since $\alpha(z_{k_1})=\alpha_0$, this proves the strict inequality
\eqref{eq:trapping-condition} at $z_{k_1}$.  But the definition of $k_1$
requires $\alpha(z_{k_1})\delta_d\leq
2\norm A R_U\norm{y(z_{k_1})}+B$, a contradiction.
Thus \eqref{eq:trapping-condition} holds at every $z_k$ with $k\geq k_0$.
Lemmata~\ref{lem:trapping-condition} and~\ref{lem:reduction} now give
$u_k\in F_d$ and $\alpha(z_k)=\alpha_0$ for all $k\geq k_0$; the same
inductive estimate for $y(z_k)$ holds for all such $k$.

We further show that
\begin{equation}\label{eq:delta-upper}
  \delta_d\leq\norm A R_U.
\end{equation}
To this end, choose any $w\in U\setminus F_d$ attaining the
minimum that defines $\delta_d$; then
\[
  \delta_d=\ip{Ad}{w}
  \leq\norm{Ad}\norm w\leq\norm A R_U
\]
because $d\in N_0\subseteq\Sph$, hence $\norm d=1$.  Hence
$C_d\eta_d=\delta_d/(4\norm A R_U)\leq1/4$, and
\begin{align*}
  \norm{z_k}
  &\leq\alpha_0+\norm{y(z_k)}\\
  &\leq(1+C_d\eta_d)\alpha_0+C_d(1+B)\\
  &\leq\frac54\cdot\frac{3\rho}{2}+C_d(1+B)
   <2\rho+C_d(1+B).
\end{align*}

It remains to check the special case $F_d=U$.  Every increment already lies
in $F_d$, so Lemma~\ref{lem:reduction} applies for all future times without
using a gap $\delta_d$.  With the choice
$\eta_d=1/(4C_d)$,
\begin{align*}
  \norm{z_k}
  &\leq\alpha_0+C_d(1+\eta_d\alpha_0+B)\\
  &=\frac54\alpha_0+C_d(1+B)\\
  &\leq\frac{15}{8}\rho+C_d(1+B)
   <2\rho+C_d(1+B).
\end{align*}
This proves \eqref{eq:trapping-bound} in both cases.
\end{proof}

The constants in Proposition~\ref{prop:trapping} depend on $d$.  We next
cover $N_0$ by finitely many neighborhoods on the unit sphere in which
that proposition applies.  Taking the maxima of the corresponding local
constants will give one bound for every trajectory whose state has its
unit direction in this union and whose norm exceeds a common threshold.

\begin{corollary}[finite cover]\label{cor:finite-cover}
Assume the inductive hypothesis.  There are an open set
$W\subseteq\Sph$, with $N_0\subseteq W\neq\Sph$, and constants
$K\geq0$, $C_W\geq1$ such that whenever
\[
  \rho:=\norm{z_{k_0}}>2K(1+B),
  \qquad \frac{z_{k_0}}\rho\in W,
\]
one has
\[
  \norm{z_k}\leq2\rho+C_W(1+B)\qquad\text{for all integers }k\geq k_0.
\]
When $N_0=\varnothing$, one may take $W=\varnothing$, $K=0$, and
$C_W=1$; no state then satisfies the entrance assumption
$z_{k_0}/\rho\in W$.
\end{corollary}

The case $N_0=\varnothing$ means that there are no degenerate unit
directions: $H$ is strictly positive on the whole unit sphere.  By
compactness, its minimum there is positive.  No face-reduction or
entrance estimate is needed in this case; the decrease estimate of
Lemma~\ref{lem:block} applies with $W=\varnothing$, and only Case~2 occurs
in the completion of the proof.  The boundedness conclusion of
Theorem~\ref{thm:main} is unchanged.

\begin{proof}
The empty case was just described.  Suppose $N_0\neq\varnothing$.  In the
nontrivial case $U\neq\{0\}$, one has $N_0\neq\Sph$.  Indeed, if
$H(\theta)=H(-\theta)=0$ for every $\theta\in\Sph$, then for every $u\in U$
both $\ip{A\theta}{u}\geq0$ and $\ip{A\theta}{u}\leq0$.  Hence
$A^{\mathsf T}u=0$.  Condition \eqref{eq:strong-positive} makes $A$ (and
therefore $A^{\mathsf T}$) invertible, so $u=0$ for all $u\in U$, a
contradiction.

Choose $\theta_*\in\Sph\setminus N_0$.  For every $d\in N_0$, let
$\varepsilon_d>0$ be the radius from Proposition~\ref{prop:trapping} and
define
\[
  r_d:=\min\left\{\varepsilon_d,\frac12\norm{d-\theta_*}\right\}>0,
  \qquad
  V_d:=\{\theta\in\Sph:\norm{\theta-d}<r_d\}.
\]
The set $V_d$ is an open neighborhood of $d$ on the unit sphere, called a
\emph{spherical cap}.  Since $r_d\leq\varepsilon_d$, every
$\theta\in V_d$ satisfies $\norm{\theta-d}<\varepsilon_d$, which is the
condition on the state's unit direction in Proposition~\ref{prop:trapping}.
Moreover, $\theta_*\notin V_d$ because
$r_d\leq\norm{d-\theta_*}/2$.

For every $d\in N_0$, we have $d\in V_d$, so the family
$\{V_d:d\in N_0\}$ is an open cover of $N_0$.
By compactness of $N_0$ (Claim~\ref{clm:N0-compact}), there are finitely
many directions $d_1,\ldots,d_m\in N_0$ such that
\[
  N_0\subseteq\bigcup_{i=1}^m V_{d_i}.
\]
Set $W:=\bigcup_{i=1}^m V_{d_i}$.  This union is open in $\Sph$, contains
$N_0$, and omits $\theta_*$; hence $W\neq\Sph$.  The selected neighborhoods
may overlap and need not have the same radius.  Define
\[
  K:=\max_{1\leq i\leq m}K_{d_i},
  \qquad
  C_W:=\max_{1\leq i\leq m}C_{d_i}.
\]
These are maxima over finitely many selected directions, not over all of
$N_0$.  The cover and the constants are fixed once $(E,U,A)$ is fixed and
are independent of the trajectory and of the unit direction $z_{k_0}/\rho$.

To verify the claimed bound, suppose that a $B$-approximate trajectory
satisfies $\rho:=\norm{z_{k_0}}>2K(1+B)$ and $z_{k_0}/\rho\in W$.
By the definition of $W$, there is an index $i\in\{1,\ldots,m\}$ such that
$z_{k_0}/\rho\in V_{d_i}$.  Thus
\[
  \left\lVert\frac{z_{k_0}}\rho-d_i\right\rVert<r_{d_i}\leq\varepsilon_{d_i},
  \qquad
  \rho>2K(1+B)\geq2K_{d_i}(1+B).
\]
The two assumptions on $z_{k_0}$ in Proposition~\ref{prop:trapping} hold
with $d=d_i$, so it gives
\[
  \norm{z_k}\leq2\rho+C_{d_i}(1+B)
  \leq2\rho+C_W(1+B)
  \qquad\text{for all integers }k\geq k_0.
\]
This proves the uniform statement.
\end{proof}

\subsection{Decrease away from degenerate directions}
\label{subsec:case-nondegenerate}

This subsection establishes a decrease estimate for trajectory segments
whose unit directions remain outside the neighborhood of $N_0$ constructed
in Subsection~\ref{subsec:case-degenerate}.  We first introduce
space-time normalized limit paths and derive the corresponding
continuous-time decrease of $H$.  We then transfer that decrease to the
original discrete segments.  The estimate concerns these segments only;
it does not require the entire trajectory to avoid the neighborhood.
It will therefore apply in both entrance scenarios in
Subsection~\ref{subsec:completion}.

\subsubsection{Limit paths}
\label{subsubsec:limit-paths}

Consider a $B$-approximate trajectory $(z_k)$.  To prove the boundedness
property in Theorem~\ref{thm:main}, we must rule out that such trajectories
reach radii arbitrarily large relative to $1+\norm{z_0}+B$.  Since every
increment has norm at most $R_U$, changing the norm of the state by at least
$\delta R$ requires at least $\delta R/R_U$ updates.  It is therefore
necessary to control segments containing a number of updates proportional to
the radius $R$.

For this purpose, we introduce space-time normalized curves.  We first join consecutive
states of a trajectory by straight line segments.  For every $j$, we then
choose a possibly different trajectory $(z_k^j)_{k\geq0}$, a scale $R_j$,
and an initial index $q_j$.  The trajectories may have different initial
states and may make different admissible choices whenever the minimizing
increment is nonunique.  If instead the initial state and a deterministic
selection rule are fixed, the recursion determines a single trajectory; in
that case one may take $(z_k^j)=(z_k)$ for every $j$ and vary only $R_j$ and
$q_j$.  We shift $q_j$ to time zero and divide both elapsed discrete time and
the states by $R_j$.  If $s$ denotes the original interpolation time, the
normalized time is $t=(s-q_j)/R_j$, equivalently $s=q_j+R_jt$.
Thus $\lceil TR_j\rceil$ updates are sufficient to represent the
interpolated trajectory on the fixed time interval $[0,T]$, while a state
with norm at most $2R_j$ is represented by a point with norm at most $2$.
The following definition makes this construction precise.

\begin{definition}[paths, interpolation, and normalized limit paths]
\label{def:interpolation}
Let $J\subseteq\mathbb R$ be a nonempty interval.  A \emph{path in $E$ on
$J$} is a continuous map $\gamma\colon J\to E$.

Let $(z_k)_{k\geq0}$ be an infinite trajectory with increments $(u_k)$.  Its
\emph{piecewise affine interpolation}, or \emph{interpolated path}, is the
path $\widetilde z\colon[0,\infty)\to E$ defined, for every $k\geq0$ and
every $t\in[k,k+1]$, by
\begin{equation}\label{eq:interpolation}
  \widetilde z(t)
  :=(k+1-t)z_k+(t-k)z_{k+1}
  =z_k+(t-k)u_k.
\end{equation}
In particular, $\widetilde z(k)=z_k$ for every $k\geq0$.  Since
$\norm{u_k}\leq R_U$, the interpolation is $R_U$-Lipschitz.

For each $j\geq1$, let $(z^j_k)_{k\geq0}$ be an infinite trajectory.  These
trajectories need not share an initial state or a rule for resolving
nonunique minimizing increments, but they are allowed to be repeated copies
of the same trajectory.  Given positive scales $R_j$, indices
$q_j\in\mathbb N_0$, and a nonempty compact interval
$J\subseteq[0,\infty)$, define the space-time normalized interpolated paths
\[
  \widehat z_j\colon J\to E,
  \qquad
  \widehat z_j(t):=R_j^{-1}\widetilde z^j(q_j+R_jt).
\]
Any uniform limit on $J$ of a subsequence of
$(\widehat z_j)$ is called a \emph{normalized limit path}.  If, in addition,
$(z^j_k)$ is $B_j$-approximate, $\lim_{j\to\infty}R_j=\infty$, and
$\lim_{j\to\infty}B_j/R_j=0$, we call
such a limit simply a \emph{limit path}.
\end{definition}

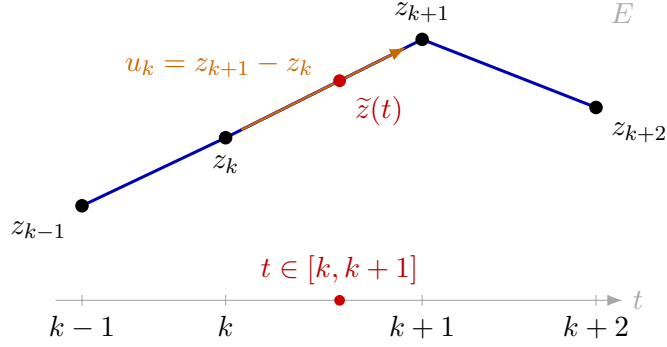
\begin{figure}[H]
\centering
\begin{tikzpicture}[
  >=Latex,
  state/.style={circle,fill=black,inner sep=1.8pt},
  path/.style={blue!70!black,line width=1.15pt},
  increment/.style={-{Latex[length=2.2mm]},orange!80!black,line width=0.9pt}
]
  \coordinate (zkm) at (0.3,0.8);
  \coordinate (zk)  at (2.2,1.7);
  \coordinate (zkp) at (4.8,3.0);
  \coordinate (zkpp) at (7.1,2.1);
  \coordinate (zt) at ($(zk)!0.58!(zkp)$);

  \node[gray!65] at (7.45,3.35) {$E$};
  \draw[path] (zkm)--(zk)--(zkp)--(zkpp);
  \draw[increment] ($(zk)!0.08!(zkp)$)--($(zk)!0.92!(zkp)$)
    node[midway,above left] {$u_k=z_{k+1}-z_k$};

  \node[state,label=below left:$z_{k-1}$] at (zkm) {};
  \node[state,label=below:$z_k$] at (zk) {};
  \node[state,label=above:$z_{k+1}$] at (zkp) {};
  \node[state,label=below right:$z_{k+2}$] at (zkpp) {};
  \filldraw[red!80!black] (zt) circle (2.2pt);
  \node[red!70!black,below right=2pt] at (zt) {$\widetilde z(t)$};

  \coordinate (axisref) at (0,-0.45);
  \draw[gray!70,-{Latex[length=2mm]}] (-0.05,-0.45)--(7.45,-0.45)
    node[right] {$t$};
  \foreach \point/\lab in {zkm/{k-1},zk/k,zkp/{k+1},zkpp/{k+2}} {
    \draw[gray!70] (\point |- axisref) ++(0,-0.07)--++(0,0.14);
    \node[below=2pt] at (\point |- axisref) {$\lab$};
  }
  \coordinate (tau) at (zt |- axisref);
  \filldraw[red!80!black] (tau) circle (1.8pt);
  \node[above=2pt,red!70!black] at (tau) {$t\in[k,k+1]$};
\end{tikzpicture}
\caption{The piecewise affine interpolation from
\eqref{eq:interpolation}.  At integer times it passes through the discrete
states $z_k$; between $k$ and $k+1$ it traverses the segment with constant
velocity $u_k$, so that $\widetilde z(t)=z_k+(t-k)u_k$.}
\label{fig:piecewise-affine-path}
\end{figure}

We next prove the properties of these normalized paths that will be needed for
the original trajectories.  Lemma~\ref{lem:limit} shows that, when
$\lim_{j\to\infty}R_j=\infty$, $\lim_{j\to\infty}B_j/R_j=0$, and the normalized states remain bounded, the
paths have a uniformly convergent subsequence.  The limiting path evolves
within the convex hull of the exact minimizing updates at its current state.
Lemma~\ref{lem:envelope} then identifies the derivative of $H$ along this
limit with the negative quadratic form induced by $A$.  In the second part of
this subsection, these results are transferred back to the original discrete
trajectories.  Consider segments starting at $q_j$ and ending after
$\lfloor TR_j\rfloor$ updates.  The bounds on the state norms and the
condition on their unit directions in Lemma~\ref{lem:block} are required
for every $q_j\leq k\leq q_j+\lceil TR_j\rceil$.  Under these conditions,
the decrease of $H$, divided by $R_j$, has limes inferior at least $cT$,
where $c>0$ is defined in \eqref{eq:aprime}.  The completion of the induction
uses this estimate to rule out trajectories reaching radii arbitrarily
large relative to $1+\norm{z_0}+B$.

\begin{lemma}[compactness of approximate trajectories]\label{lem:limit}
Let $(R_j)_{j\geq1}$ and $(B_j)_{j\geq1}$ be sequences of positive and
nonnegative real numbers, respectively, such that
\[
  \lim_{j\to\infty}R_j=\infty,
  \qquad \lim_{j\to\infty}\frac{B_j}{R_j}=0.
\]
For each integer $j\geq1$, let $(z^j_k)_{k\in\mathbb N_0}$ be a
$B_j$-approximate trajectory.  Suppose that there exist $T>0$ and a
sequence $(q_j)_{j\geq1}$ in $\mathbb N_0$ such that, for every $j\geq1$,
\[
  \norm{z^j_k}\leq2R_j
  \quad\text{for all integers }k\text{ with }
  q_j\leq k\leq q_j+\lceil TR_j\rceil.
\]
Then the paths
\[
  \widehat z_j(t):=R_j^{-1}\widetilde z^j(q_j+R_jt),
  \qquad\text{for all }t\in[0,T],
\]
have a uniformly convergent subsequence.  Every resulting limit path $z$ is
$R_U$-Lipschitz and satisfies
\begin{equation}\label{eq:limit-inclusion}
  \dot z(t)\in\conv\mathcal F(z(t))
\end{equation}
for almost every $t\in[0,T]$.\footnote{Throughout, ``almost every'' and
``almost everywhere'' mean that the exceptional set has Lebesgue measure zero.}
\end{lemma}

\begin{proof}
The normalized paths are uniformly bounded by $2$ in norm and are
$R_U$-Lipschitz, so the
finite-dimensional Arzel\`a--Ascoli theorem
(Theorem~\ref{thm:appendix-aa}) applies.  Let $z$ be a uniform limit.
By Theorem~\ref{thm:appendix-lipschitz}, the Lipschitz path $z$ is
differentiable at every $t\in(0,T)$ except for a set of Lebesgue measure
zero.  Fix such a point $t$ and set $x:=z(t)$.  If
$\mathcal F(x)=U$, then, for every $h>0$ with $t+h\leq T$, the quotient
\[
  \frac{\widehat z_j(t+h)-\widehat z_j(t)}h
\]
belongs to $\conv U$: it is a convex combination of the slopes on the grid
intervals traversed between $t$ and $t+h$.  Uniform convergence and the
closedness of $\conv U$ give the same conclusion for the corresponding
difference quotient of $z$.  Letting $h\downarrow0$ proves
\eqref{eq:limit-inclusion} at $t$.  We may therefore suppose that
$\mathcal F(x)\neq U$.  Let
\[
  r:=\frac{g(x)}{4\norm A R_U}>0
\]
and choose $h>0$ such that $t+h\leq T$ and
\begin{equation}\label{eq:short-limit-interval}
  R_Uh<\frac r2.
\end{equation}
For $s\in[t,t+h]$, Lipschitz continuity gives
$\norm{z(s)-x}\leq R_Uh<r/2$.  Uniform convergence, together with the fact
that the normalized mesh size is $R_j^{-1}$, therefore gives an index $j_1$
such that, for every $j\geq j_1$, the left endpoint of every normalized grid
interval meeting $[t,t+h]$ lies within distance $r$ of $x$.  Since
$\lim_{j\to\infty}B_j/R_j=0$, we may increase $j_1$ so that the normalized optimality error
at every such grid point is smaller than $g(x)/2$ whenever $j\geq j_1$.
Claim~\ref{clm:face-stability} then implies that
every increment used by $\widehat z_j$ on $[t,t+h]$ belongs to
$\mathcal F(x)$.  Hence
\[
  \frac{\widehat z_j(t+h)-\widehat z_j(t)}h
  \in\conv\mathcal F(x).
\]
Letting $j\to\infty$ gives the same inclusion for the difference quotient of
$z$.  Since arbitrarily small positive $h$ satisfy
\eqref{eq:short-limit-interval}, letting $h\downarrow0$ yields
\eqref{eq:limit-inclusion}.
\end{proof}

The differentiation of a finite maximum used below is the classical
envelope principle of Danskin \cite{Danskin1966}.  Hofbauer and Sandholm
state its almost-everywhere form for finitely many Lipschitz functions
\cite[Theorem~A.4]{HofbauerSandholm2009}.  A closely related calculation
for the gap along best-reply dynamics appears in
Hadikhanloo et al.\ \cite[Eq.~(26)]{HadikhanlooEtAl2022}.  We include the argument because the
precise identity below also uses the convexified differential inclusion and
must hold almost everywhere along a merely Lipschitz path.

\begin{lemma}[finite-action envelope identity]\label{lem:envelope}
Let $z\colon[0,T]\to E$ be Lipschitz and satisfy
\eqref{eq:limit-inclusion}.  Then $H\circ z$ is Lipschitz and
\begin{equation}\label{eq:envelope}
  \frac{d}{dt}H(z(t))
  =-\ip{A\dot z(t)}{\dot z(t)}
  \leq-\mu\norm{\dot z(t)}^2
\end{equation}
for almost every $t$.
\end{lemma}

\begin{proof}
The Lipschitz assertion follows from Claim~\ref{clm:H-properties}.  For
$F\subseteq U$, let
$E_F:=\{t\in[0,T]:\mathcal F(z(t))=F\}$.  At every time $t$, the
finite nonempty set $U$ has a unique nonempty set of minimizers
$\mathcal F(z(t))$.  Thus $t$ belongs to exactly one of the sets $E_F$:
their union is $[0,T]$, and distinct such sets are disjoint.  The nonempty
sets $E_F$ therefore form a finite partition of $[0,T]$.  Each is
measurable, since membership is specified by the equalities
$\ip{Az(t)}{u}+H(z(t))=0$ for $u\in F$ and the strict inequalities
$\ip{Az(t)}{u}+H(z(t))>0$ for $u\in U\setminus F$, all involving
continuous functions of $t$.
By Theorems~\ref{thm:appendix-lipschitz}
and~\ref{thm:appendix-density}, for almost every $t$ the path is
differentiable, \eqref{eq:limit-inclusion} holds, and $t$ is a density point
of the member $E_F$ containing it.  Fix such a $t\in(0,T)$ and write
$F=\mathcal F(z(t))$.

Fix $u,w\in F$.  For every $s\in E_F$, both $u$ and $w$ minimize the
score at $z(s)$, so $\ip{Az(s)}{u-w}=0$.
Because $t$ is a density point of $E_F$, there are nonzero numbers $h_n$
with $\lim_{n\to\infty}h_n=0$ such that $t+h_n\in E_F$.  The score difference vanishes both
at $t$ and at $t+h_n$, and hence
\[
  0=\frac{\ip{Az(t+h_n)}{u-w}-\ip{Az(t)}{u-w}}{h_n}
   =\ip{A\frac{z(t+h_n)-z(t)}{h_n}}{u-w}.
\]
Since $z$ is differentiable at $t$, letting $n\to\infty$ yields
$\ip{A\dot z(t)}{u-w}=0$.  Thus the numbers
$\ip{A\dot z(t)}{u}$ have a common value, denoted by $c_t$, for all
$u\in F$.  This argument does not require $E_F$ to contain an interval
around $t$; the sequence provided by the density-point property suffices.

Next, \eqref{eq:limit-inclusion} gives coefficients $\lambda_u\geq0$,
for $u\in F$, such that
$\sum_{u\in F}\lambda_u=1$ and
$\dot z(t)=\sum_{u\in F}\lambda_u u$.  Consequently,
\[
  \ip{A\dot z(t)}{\dot z(t)}
  =\sum_{u\in F}\lambda_u\ip{A\dot z(t)}{u}
  =\sum_{u\in F}\lambda_u c_t=c_t.
\]

It remains to identify the derivative of $H\circ z$.
Continuity of $z$ and Claim~\ref{clm:face-stability}, applied with
$x=z(t)$ and $\beta=0$, give $\mathcal F(z(s))\subseteq F$ for all
$s$ sufficiently close to $t$; if $F=U$, this inclusion is automatic.
Thus, near $t$,
\[
  H(z(s))=\max_{u\in F}\ip{-Az(s)}{u}.
\]
For every $u\in F$, differentiability of $z$ gives the same constant and
linear terms in the expansion
\[
  \ip{-Az(t+h)}{u}=H(z(t))-c_t h+r_u(h),
  \qquad \lim_{h\to0}\frac{|r_u(h)|}{|h|}=0.
\]
Since $F$ is finite, $\lim_{h\to0}\max_{u\in F}|r_u(h)|/|h|=0$ as well.
Taking the maximum therefore gives
$H(z(t+h))=H(z(t))-c_t h+o(|h|)$, as in
Lemma~\ref{lem:appendix-finite-max}.  Hence $H\circ z$ is differentiable
at $t$, and the preceding calculation of $c_t$ yields
\[
  \frac{d}{dt}H(z(t))=-c_t
  =-\ip{A\dot z(t)}{\dot z(t)}.
\]
Finally, \eqref{eq:strong-positive}, applied to $\dot z(t)$, bounds this
derivative above by $-\mu\norm{\dot z(t)}^2$.
The points excluded at the start of the argument form a set of Lebesgue
measure zero.  Thus both conclusions in \eqref{eq:envelope} hold almost
everywhere on $[0,T]$.
\end{proof}

\subsubsection{Discrete decrease estimate}
\label{subsubsec:block-decrease}

We next prove an asymptotic decrease estimate for sequences of trajectory
segments whose unit directions remain outside a fixed neighborhood of
$N_0\subseteq\Sph$.  The segments considered in Lemma~\ref{lem:block}
have lengths proportional to $R_j$, state norms between
$\varepsilon R_j$ and $2R_j$, and additive oracle errors $B_j$ from
Definition~\ref{def:B-trajectory} satisfying
$\lim_{j\to\infty}B_j/R_j=0$, where $\lim_{j\to\infty}R_j=\infty$.
The limit-path results give a uniform negative
rate of change of $H$ along their normalized limits.  The lemma transfers
this estimate to the discrete segments: their decrease in $H$, divided
by $R_j$, has a strictly positive limes inferior.  This does not assert
decrease at every update.  Both entrance scenarios in the completion of
the induction provide segments satisfying the lemma's assumptions.

Let $W\subseteq\Sph$ be open, contain $N_0$, and not equal $\Sph$; when
$N_0=\varnothing$, we allow $W=\varnothing$.  Compactness gives
\begin{equation}\label{eq:aprime}
  a':=\min_{\theta\in\Sph\setminus W}H(\theta)>0,
  \qquad
  c:=\mu\left(\frac{a'}{\norm A}\right)^2>0.
\end{equation}

To explain the constants in \eqref{eq:aprime}, consider trajectory segments
starting at indices $q_j$ and ending after $\lfloor TR_j\rfloor$ updates.
For a fixed $\varepsilon>0$, assume that the state norms stay between
$\varepsilon R_j$ and $2R_j$ and that their unit directions stay outside
$W$, for every index $q_j\leq k\leq q_j+\lceil TR_j\rceil$.
Since $H$ is continuous and is strictly
positive on the compact set $\Sph\setminus W$, the first constant in
\eqref{eq:aprime} satisfies $a'>0$.  Claim~\ref{clm:directional-bound} then implies
$\norm{\dot z(t)}\geq a'/\norm A$ for almost every $t$ along every limit
path obtained from such segments.  With the second constant in
\eqref{eq:aprime}, Lemma~\ref{lem:envelope} yields
$(H\circ z)'(t)\leq-c$ almost everywhere.  The next lemma transfers this
continuous-time estimate to the original segments and shows that their
decrease in $H$, divided by $R_j$, has limes inferior at least $cT$.

\Needspace{14\baselineskip}
\begin{lemma}[segment decrease]\label{lem:block}
Let $(R_j)_{j\geq1}$ and $(B_j)_{j\geq1}$ be sequences of positive and
nonnegative real numbers, respectively, such that
\[
  \lim_{j\to\infty}R_j=\infty,
  \qquad \lim_{j\to\infty}\frac{B_j}{R_j}=0.
\]
For each integer $j\geq1$, let $(z^j_k)_{k\in\mathbb N_0}$ be a
$B_j$-approximate trajectory.  Suppose that there exist $\varepsilon>0$,
$T>0$, and a sequence $(q_j)_{j\geq1}$ in $\mathbb N_0$ such that, for
every $j\geq1$,
\[
  \varepsilon R_j\leq\norm{z^j_k}\leq2R_j,
  \qquad
  \frac{z^j_k}{\norm{z^j_k}}\notin W
\]
for all integers $k$ with $q_j\leq k\leq q_j+\lceil TR_j\rceil$.  Then
\[
 \liminf_{j\to\infty}
 \frac{H(z^j_{q_j})-
 H(z^j_{q_j+\lfloor TR_j\rfloor})}{R_j}
 \geq cT.
\]
\end{lemma}

\begin{proof}
Set
\[
 \Delta_j:=\frac{H(z^j_{q_j})-
 H(z^j_{q_j+\lfloor TR_j\rfloor})}{R_j}.
\]
Because $H$ is $b$-Lipschitz, where $b=\norm A R_U$, and each update has
norm at most $R_U$, we have
\[
 |\Delta_j|\leq\frac{bR_U\lfloor TR_j\rfloor}{R_j}\leq bR_UT.
\]
Thus $L:=\liminf_{j\to\infty}\Delta_j$ is finite.  By the definition of
the limes inferior, for every integer $m\geq1$ all sufficiently late
terms exceed $L-1/m$, and terms smaller than $L+1/m$ occur at arbitrarily
large indices.  We can therefore choose strictly increasing indices
$j_m$ with $|\Delta_{j_m}-L|<1/m$, so that
$\lim_{m\to\infty}\Delta_{j_m}=L$.  This is convergence along a
subsequence, not attainment of $L$ by a term of the sequence.
Apply Lemma~\ref{lem:limit} along this subsequence and take a further
subsequence on which the normalized paths converge uniformly.  The
scalar limit remains $L$; we relabel the resulting subsequence by $j$.
The $O(R_j^{-1})$ distance between the interpolation
and adjacent grid points shows that the limit obeys
$\norm{z(t)}\geq\varepsilon$ and
$z(t)/\norm{z(t)}\in\Sph\setminus W$.  By
Claim~\ref{clm:directional-bound},
\[
  \norm{\dot z(t)}\geq\frac{a'}{\norm A}
  \quad\text{for almost every }t.
\]
Lemma~\ref{lem:envelope} yields $(H\circ z)'\leq-c$.  Application of the
fundamental theorem of integration
(Theorem~\ref{thm:appendix-lipschitz}) gives
\begin{equation}\label{eq:limit-path-decrease}
  H(z(0))-H(z(T))\geq cT.
\end{equation}
To identify the discrete endpoint, set
\[
  t_j:=\frac{\lfloor TR_j\rfloor}{R_j}.
\]
Then $0\leq T-t_j<R_j^{-1}$ and
$\widehat z_j(t_j)=R_j^{-1}z^j_{q_j+\lfloor TR_j\rfloor}$.  Since
$\widehat z_j$ is $R_U$-Lipschitz and $H$ is $b$-Lipschitz,
\[
 \left|
   \frac{H(z^j_{q_j+\lfloor TR_j\rfloor})}{R_j}
   -H(\widehat z_j(T))
 \right|
 \leq \frac{bR_U}{R_j}.
\]
The right-hand side has limit zero because
$\lim_{j\to\infty}R_j=\infty$.
Likewise, positive homogeneity gives
$R_j^{-1}H(z^j_{q_j})=H(\widehat z_j(0))$.  Uniform convergence and
continuity of $H$ therefore show that, along the subsequence chosen above,
the normalized discrete decrease converges to
$H(z(0))-H(z(T))$.  Hence $L=H(z(0))-H(z(T))\geq cT$ by
\eqref{eq:limit-path-decrease}, which proves the
claim.
\end{proof}

\subsection{Completion of the induction}\label{subsec:completion}

We combine Corollary~\ref{cor:finite-cover} and Lemma~\ref{lem:block}
according to the two-case construction illustrated in
Figure~\ref{fig:proof-overview} in
Subsection~\ref{subsec:proof-technique-intro}.

\begin{proof}[Completion of the proof of Theorem~\ref{thm:main}]
Suppose, for a contradiction, that the conclusion fails for $(E,U,A)$.  The
negation of the claimed bound
says that, for every integer $j\geq1$, there are $B_j\geq0$, a
$B_j$-approximate trajectory $(z^j_k)$, and an index at which
\[
  \norm{z^j_k}>j(1+\norm{z^j_0}+B_j).
\]
Set
\[
  R_j:=j\rho_j,
  \qquad \rho_j:=1+\norm{z^j_0}+B_j,
\]
and let $k_j$ be the first index with $\norm{z^j_{k_j}}\geq R_j$.  In
particular,
\begin{equation}\label{eq:scales}
  \lim_{j\to\infty}R_j=\infty,\qquad
  \lim_{j\to\infty}\frac{\rho_j}{R_j}=0,
  \qquad \lim_{j\to\infty}\frac{B_j}{R_j}=0.
\end{equation}

Choose $W,K,C_W$ as in Corollary~\ref{cor:finite-cover}, and $a',c$ as in
\eqref{eq:aprime}.  Put
\[
  b:=\norm A R_U,\qquad b_1:=bR_U,\qquad
  T_0:=\frac1{8R_U},
\]
and choose
\begin{equation}\label{eq:epsilon}
  0<\varepsilon\leq
  \min\left\{\frac1{12},\frac{cT_0}{4b}\right\}.
\end{equation}

Since $\rho_j/R_j=1/j$, $\norm{z_0^j}\leq\rho_j$, and $R_j\geq j$, choose
an integer $j_0$ satisfying
\[
  j_0>\max\left\{
    \frac{2K}{\varepsilon}, 3C_W, \frac1\varepsilon, 16R_U
  \right\}.
\]
Then, for every $j\geq j_0$,
\begin{equation}\label{eq:j0-conditions}
  \varepsilon R_j>2K(1+B_j),\qquad
  C_W(1+B_j)\leq\frac{R_j}{3},\qquad
  \norm{z_0^j}\leq\varepsilon R_j,\qquad
  R_j\geq16R_U.
\end{equation}

For every $j\geq j_0$, define
\[
  \mathcal E_j^W:=\left\{
    k\in\{0,\ldots,k_j\}:
    \norm{z_k^j}\geq\varepsilon R_j,
    \ \frac{z_k^j}{\norm{z_k^j}}\in W
  \right\}.
\]
If $\mathcal E_j^W\neq\varnothing$, then it is a nonempty finite set, so its
minimum exists; in this case set $k_j^W:=\min\mathcal E_j^W$.  If
$\mathcal E_j^W=\varnothing$, set $k_j^W:=+\infty$.  Finally, put
$k^*_j:=\min\{k_j^W,k_j\}$.  We claim that, for every $j\geq j_0$,
\begin{equation}\label{eq:critical-radius}
  \norm{z^j_{k^*_j}}\geq\frac{R_j}{3}.
\end{equation}
\emph{Case~1: $k_j^W<k_j$.}  Then $k^*_j=k_j^W$,
$\mathcal E_j^W\neq\varnothing$ and $k_j^W=\min\mathcal E_j^W$.
This case is represented by the green trajectory in
Figure~\ref{fig:proof-overview}, with the entrance state
$z^j_{k_j^W}$ labeled $z_e$.  Define
\[
  R^e_j:=\norm{z^j_{k_j^W}}
  =\norm{z^j_{\min\mathcal E_j^W}}
  \geq\varepsilon R_j.
\]
By
\eqref{eq:j0-conditions}, $R^e_j>2K(1+B_j)$, and
Corollary~\ref{cor:finite-cover} gives
\[
  R_j\leq\norm{z^j_{k_j}}
  \leq2R^e_j+C_W(1+B_j)
  \leq2R^e_j+\frac{R_j}{3}.
\]
Thus $R^e_j\geq R_j/3$.

\emph{Case~2: $k_j^W\geq k_j$.}  Then $k^*_j=k_j$, and
$\norm{z^j_{k^*_j}}\geq R_j$ by the definition of $k_j$.
The red trajectory in Figure~\ref{fig:proof-overview} illustrates this
case, with the target-crossing state $z^j_{k_j}$ labeled $z_{k_R}$.
This case includes an entrance occurring exactly at the target crossing,
as well as $k_j^W=+\infty$.

Thus \eqref{eq:critical-radius} holds in both cases.  We now construct
a preceding segment whose unit directions lie outside $W$ and apply
the same decrease estimate to it.  Let
\[
  k'_j:=\max\{k\leq k^*_j:\norm{z^j_k}\leq\varepsilon R_j\},
  \qquad s_j:=k'_j+1,
  \qquad \ell_j:=k^*_j-1.
\]
The set defining $k'_j$ is nonempty for every $j\geq j_0$ by
\eqref{eq:j0-conditions}.  Moreover,
$\norm{z^j_{k^*_j}}\geq R_j/3>\varepsilon R_j$, so $k'_j<k^*_j$.  If
$s_j=k^*_j$, then the single update from $k'_j$ to $k^*_j$ would satisfy
\[
  R_U\geq\norm{z^j_{k^*_j}-z^j_{k'_j}}
  \geq\left(\frac13-\varepsilon\right)R_j,
\]
which is impossible because $\varepsilon\leq1/12$ and $R_j\geq16R_U$.
Hence $s_j\leq k^*_j-1=\ell_j$.  For every
$s_j\leq k\leq\ell_j$,
\begin{equation}\label{eq:good-segment}
  \varepsilon R_j<\norm{z^j_k}<R_j,
  \qquad \frac{z^j_k}{\norm{z^j_k}}\notin W.
\end{equation}
In Figure~\ref{fig:proof-overview}, this is the thick green or red segment
from $z_s$ to $z_\ell$; the next state, at the entrance or target crossing,
is not part of this segment.
Moreover,
\[
  \norm{z^j_{s_j}}\leq\varepsilon R_j+R_U,
  \qquad
  \norm{z^j_{\ell_j}}\geq\frac{R_j}{3}-R_U.
\]
Because every step has length at most $R_U$, \eqref{eq:epsilon} and
\eqref{eq:j0-conditions} give
\begin{align*}
  (\ell_j-s_j)R_U
  &\geq\norm{z^j_{\ell_j}-z^j_{s_j}}\\
  &\geq\norm{z^j_{\ell_j}}-\norm{z^j_{s_j}}\\
  &\geq\left(\frac13-\varepsilon\right)R_j-2R_U\\
  &\geq\frac{R_j}{8}=T_0R_UR_j.
\end{align*}
Hence $\ell_j-s_j\geq T_0R_j$ for every $j\geq j_0$.  We may therefore
write
$\ell_j-s_j\geq\lceil T_0R_j\rceil$ and apply
Lemma~\ref{lem:block} on the first $\lfloor T_0R_j\rfloor$ steps of the segment
\eqref{eq:good-segment}.  It gives
\begin{equation}\label{eq:lower-decrease}
  \liminf_{j\to\infty}
  \frac{H(z^j_{s_j})-
  H(z^j_{s_j+\lfloor T_0R_j\rfloor})}{R_j}
  \geq cT_0.
\end{equation}
On the other hand, Claim~\ref{clm:H-properties} and one bounded step from
$k'_j$ to $s_j$ give
\begin{align*}
  H(z^j_{s_j})-
  H(z^j_{s_j+\lfloor T_0R_j\rfloor})
  &\leq H(z^j_{s_j})\\
  &\leq H(z^j_{k'_j})+bR_U\\
  &\leq b\varepsilon R_j+b_1.
\end{align*}
Thus the limes superior of the quotient in \eqref{eq:lower-decrease} is at most
$b\varepsilon\leq cT_0/4$, a contradiction.  The induction is complete.
\end{proof}

\Needspace{12\baselineskip}
\subsection{Proof of Corollary~\ref{cor:weighted}}
\label{subsec:proof-weighted}

We reduce weighted updates to unit updates by recording the accumulated
weight of each vector in $U$.  Rounding these weights down to integers
leaves a bounded remainder.  We then show that the resulting unit updates
satisfy the approximate oracle condition with a uniformly enlarged error.

\begin{proof}
Let $C=C(E,U,A)\geq1$ be the constant in Theorem~\ref{thm:main}, and put
\[
  S_U:=\sum_{u\in U}\norm u,
  \qquad
  C_{\mathrm w}:=C\bigl(1+2\norm A R_US_U\bigr)+S_U.
\]
If $\overline w=0$ or $U=\{0\}$, every state equals $z_0$, so the claimed bound
holds.  First consider $\overline w=1$.  For every $u\in U$ and every state index
$k$, define
\[
  a_u(k):=\sum_{\substack{0\leq i<k\\u_i=u}}w_i,
  \qquad
  q_k:=z_0+\sum_{u\in U}\lfloor a_u(k)\rfloor u.
\]
Here $a_u(k)$ is the total weight assigned to $u$ before step $k$.
Telescoping \eqref{eq:weighted-recursion} gives
\begin{equation}\label{eq:weighted-rounding-remainder}
  z_k-q_k=\sum_{u\in U}
      \bigl(a_u(k)-\lfloor a_u(k)\rfloor\bigr)u,
  \qquad \norm{z_k-q_k}\leq S_U.
\end{equation}
Only $a_{u_k}$ changes at step $k$, and its increase is $w_k\leq1$.
Consequently,
\[
  q_{k+1}=q_k+\xi_ku_k,
  \qquad
  \xi_k:=\lfloor a_{u_k}(k)+w_k\rfloor
             -\lfloor a_{u_k}(k)\rfloor\in\{0,1\}.
\]
Keep only the updates for which $\xi_k=1$.  They form a unit-step
sequence starting at $z_0$; between retained updates, $q_k$ is unchanged.
For each retained update and every $v\in U$, the oracle condition and
\eqref{eq:weighted-rounding-remainder} imply
\begin{align*}
  \ip{Aq_k}{u_k-v}
  &=\ip{Az_k}{u_k-v}+\ip{A(q_k-z_k)}{u_k-v}\\
  &\leq B+\norm A\norm{q_k-z_k}\norm{u_k-v}\\
  &\leq B+2\norm A R_US_U.
\end{align*}
The obtained sequence is therefore a
$(B+2\norm A R_US_U)$-approximate $U$-valued trajectory or a finite
segment of one.  If only finitely many updates are retained, extend
that sequence by exact minimizers, which exist because $U$ is finite
and nonempty.  Theorem~\ref{thm:main} and
\eqref{eq:weighted-rounding-remainder} now give, for every state index $k$,
\[
  \norm{z_k}
  \leq C\bigl(1+\norm{z_0}+B+2\norm A R_US_U\bigr)+S_U.
\]

For general $\overline w>0$, apply this estimate to the states
$z_k/\overline w$, the weights $w_k/\overline w$, and the additive oracle
error $B/\overline w$.  Multiplication by $\overline w$ yields
\[
  \norm{z_k}
  \leq C\bigl(\norm{z_0}+B\bigr)
       +\overline w\bigl[C(1+2\norm A R_US_U)+S_U\bigr]
  \leq C_{\mathrm w}\bigl(\overline w+\norm{z_0}+B\bigr),
\]
since $C_{\mathrm w}\geq C$.  The argument applies to finite and infinite
recursions alike.
\end{proof}

On a finite horizon, one may take $\overline w$ to be the largest weight and $B$
to be the largest score error on that horizon.  The same constant
$C_{\mathrm w}(E,U,A)$ applies independently of the horizon length.

\subsection{Proof of Theorem~\ref{thm:explicit-symmetric}}
\label{subsec:proof-explicit-symmetric}

The proof uses symmetry to transform the score into a Euclidean inner
product.  The relative inradius then gives a direct radial decrease estimate.

\begin{proof}
The square root $A^{1/2}$ exists by the spectral theorem for symmetric
positive-definite matrices; see Horn and Johnson
\cite[Section~7.2]{HornJohnson2013}.  Let $L:=\operatorname{span}V$,
let $P_L$ be the orthogonal projection onto $L$, and set
\[
  \mathbb B_L:=\{v\in L:\norm v\leq1\},
  \qquad r=\sup\{s>0:s\mathbb B_L\subseteq\conv V\}>0.
\]
Put
$y_k:=A^{1/2}z_k$, $v_k:=A^{1/2}u_k$, and decompose
$y_k=p_k+q_k$ with $p_k:=P_Ly_k$ and
$q_k:=(\Id_E-P_L)y_k$.  In particular,
\[
  y_0=A^{1/2}z_0,\qquad
  p_0=P_Ly_0,\qquad q_0=(\Id_E-P_L)y_0.
\]
Since every $v_k$ belongs to $L$,
$q_k=q_0$ for all $k$.  Moreover,
\[
 \ip{Az_k}{u}=\ip{y_k}{A^{1/2}u}
 =\ip{p_k}{A^{1/2}u}.
\]
Because $r\mathbb B_L\subseteq\conv V$, one has
\[
 \min_{v\in V}\ip{p_k}{v}\leq-r\norm{p_k}.
\]
The approximate oracle condition therefore gives
\[
 \ip{p_k}{v_k}\leq-r\norm{p_k}+B.
\]
Consequently,
\begin{equation}\label{eq:explicit-symmetric-drift}
 \norm{p_{k+1}}^2
 \leq\norm{p_k}^2-2r\norm{p_k}+2B+R^2.
\end{equation}
Whenever $\norm{p_k}>(2B+R^2)/(2r)$, the right-hand side of
\eqref{eq:explicit-symmetric-drift} is at most $\norm{p_k}^2$.  Otherwise
$\norm{p_{k+1}}\leq\norm{p_k}+R$.  Induction yields
\[
 \sup_{k\geq0}\norm{p_k}
 \leq\max\left\{\norm{p_0},R+\frac{2B+R^2}{2r}\right\}.
\]
Since $p_k\perp q_0$ and
$\norm z\leq\lambda_{\min}^{-1/2}\norm{A^{1/2}z}$, we obtain the sharper bound
\begin{equation}\label{eq:explicit-symmetric-bound}
 \sup_{k\geq0}\norm{z_k}
 \leq
 \frac1{\sqrt{\lambda_{\min}}}
 \left(
   \norm{q_0}^2+
   \max\left\{
      \norm{p_0},
      R+\frac{2B+R^2}{2r}
   \right\}^2
 \right)^{1/2}.
\end{equation}
For the simpler estimate in Theorem~\ref{thm:explicit-symmetric}, set
$c:=R+(2B+R^2)/(2r)$.  The triangle inequality in $\R^2$ gives
\[
 \bigl(\norm{q_0}^2+(\norm{p_0}+c)^2\bigr)^{1/2}
 \leq \bigl(\norm{q_0}^2+\norm{p_0}^2\bigr)^{1/2}+c
 =\norm{y_0}+c.
\]
Since $\max\{\norm{p_0},c\}\leq\norm{p_0}+c$ and
$\norm{y_0}\leq\sqrt{\lambda_{\max}}\norm{z_0}$, this proves
\eqref{eq:explicit-symmetric-simple}.
\end{proof}

In particular, the conclusion of Theorem~\ref{thm:main} holds under these
assumptions with the explicit choice
\[
 C_{\mathrm{sym}}
 :=\max\left\{
   \sqrt{\frac{\lambda_{\max}}{\lambda_{\min}}},
   \frac{R+R^2/(2r)}{\sqrt{\lambda_{\min}}},
   \frac1{r\sqrt{\lambda_{\min}}}
 \right\}.
\]

\subsection{Why coercivity is necessary}\label{subsec:skew-counterexample}

The coercivity assumption cannot be dropped, even when $A$ is
skew-symmetric.  Consider
\[
 E=\R^2,\qquad
 A=\begin{pmatrix}0&-1\\1&0\end{pmatrix},\qquad
 U=\{\pm e_1,\pm e_2\}.
\]
Here $0\in\conv U$ but $\ip{Az}{z}=0$ for every $z$.  Starting at the origin,
choose first $u_0=-e_2$ and thereafter choose an exact minimizer of
$u\mapsto\ip{Az_k}{u}$, resolving corner ties outwards.  The first two updates
reach the corner $(-1,-1)$.  For example, at the corner $(r,-r)$ choose
$-e_2$;
the subsequent unique minimizers move left until the trajectory reaches
$(-r-1,-r-1)$.  Apply the rotated version of the same rule at the other
corners.  The transition from a corner of $\ell_\infty$-radius $r$ to the
next corner of radius $r+1$ takes $2r+2$ steps.  Hence, after $n$ corner
transitions, the radius is of order $n$ and the elapsed time is of order
$n^2$.  Consequently
\[
 \norm{z_k}=\Theta(\sqrt{k})
\]
along this trajectory, and in particular it is unbounded.  Other
tie-breaking choices can produce a bounded cycle; the example shows that a
bound uniform over all admissible selections fails once coercivity is
dropped.

For the cube-vertex example in Figure~\ref{fig:cube-skew}, take the same
matrix $A$, but $U=\{-1,1\}^2$ and $z_0=(3/4,2/5)$.
Both state coordinates retain nonzero fractional parts, so the exact
minimizer is always unique.  More explicitly, for every $m\in\mathbb N_0$,
\[
  z_{k_m}=(4m+7/4,-3/5),\qquad k_m=8m^2+6m+1.
\]
Indeed, $z_1=(7/4,-3/5)$.  Starting at $z_{k_m}$, the oracle successively
selects $(-1,-1)$, $(-1,1)$, $(1,1)$, and $(1,-1)$ for
$4m+2$, $4m+3$, $4m+4$, and $4m+5$ updates, respectively.
These $16m+14=k_{m+1}-k_m$ updates end at
$(4(m+1)+7/4,-3/5)$, proving the formula by induction.
In particular, the trajectory in panel (c) is unbounded.

\section{Applications}\label{sec:applications}

This section develops two applications of the cycling theorem.
Subsection~\ref{subsec:affine-fw-app} proves last-iterate convergence of the Frank--Wolfe
algorithm for affine monotone VIs on polytopes. We give a further application for quadratic saddle-point problems.
Subsection~\ref{subsec:oblique-relaxation} proves boundedness of oblique
relaxation and finite termination of coordinate corrections for linear
inequalities.  

\subsection{Affine variational inequalities with a relative-interior solution}
\label{subsec:affine-fw-app}
Let us consider~\eqref{eq:affine-fw} and~\eqref{eq:affine-fw-oracle} under the stated assumptions
of Theorem~\ref{thm:app-affine-fw}.
\begin{proof}[Proof of Theorem~\ref{thm:app-affine-fw}]
Let $P_K$ be the orthogonal projection onto $T_K$ and define
\[
 U_*:=\{s-x^*:s\in\operatorname{vert}(K)\}\subset T_K,
 \qquad A_*:=P_KL|_{T_K}.
\]
Since $x^*\in K=\conv\operatorname{vert}(K)$, one has
$0\in\conv U_*$.  Moreover, for every $h\in T_K$,
\[
 \ip{A_*h}{h}=\ip{Lh}{h}\geq\mu\norm h^2.
\]

Because $x^*\in\operatorname{relint}K$, for every $h\in T_K$ there is
$t_h>0$ such that $x^*\pm t_hh\in K$.  Applying the variational inequality
to these two points gives
\begin{equation}\label{eq:interior-orthogonality}
 \ip{\Phi(x^*)}{h}=0\qquad\text{for all }h\in T_K.
\end{equation}
Set
\[
 z_k:=k(x_k-x^*)\in T_K,
 \qquad u_k:=s_k-x^*\in U_*.
\]
The harmonic update gives the exact additive identity
\[
 z_{k+1}=z_k+u_k,
 \qquad z_0=0.
\]
For $u=s-x^*\in U_*$ and $k\geq1$, affinity and
\eqref{eq:interior-orthogonality} give
\begin{align*}
 k\ip{\Phi(x_k)}{u}
 &=k\ip{\Phi(x^*)+L(x_k-x^*)}{u}\\
 &=\ip{Lz_k}{u}
 =\ip{A_*z_k}{u}.
\end{align*}
Multiplying \eqref{eq:affine-fw-oracle} by $k$ therefore yields
\[
 \ip{A_*z_k}{u_k}
 \leq\min_{u\in U_*}\ip{A_*z_k}{u}+k\delta_k
 \leq\min_{u\in U_*}\ip{A_*z_k}{u}+\overline B.
\]
At $k=0$, every element of $U_*$ minimizes the zero score.  Thus $(z_k)$ is
a $\overline B$-approximate trajectory for $(T_K,U_*,A_*)$.
Theorem~\ref{thm:main} gives
\[
 \sup_{k\geq0}\norm{z_k}\leq C(1+\overline B).
\]
Dividing by $k$ proves the claimed estimate for $x_k-x^*$.
\end{proof}

The error requirement here concerns the transformed scores: it is
$\sup_{k\geq1}k\delta_k<\infty$, not merely boundedness of the original
errors $\delta_k$.  For example, suppose that at each iteration the
available score of every vertex $s$ differs from $\ip{\Phi(x_k)}{s}$ by
at most $\varepsilon_k\geq0$, and the selected vertex minimizes the
available scores up to an error $\eta_k\geq0$.  The score comparison in
Appendix~\ref{app:oracle-errors} shows that one may take
$\delta_k=\eta_k+2\varepsilon_k$.
Thus the harmonic $O(k^{-1})$ norm bound applies if
$\sup_{k\geq1}k(\eta_k+2\varepsilon_k)<\infty$.
A fixed positive tolerance in the original scores is not covered by
this rate guarantee.

The weighted cycling corollary also gives rates when the numerator and
offset of the reciprocal step size are changed.

\begin{proposition}[Other reciprocal step sizes and oracle accuracies]
\label{prop:affine-fw-schedules}
Under the assumptions on $K$, $L$, and $x^*$ in
Theorem~\ref{thm:app-affine-fw}, fix $\beta\geq\alpha>0$ and replace
\eqref{eq:affine-fw} by
\begin{equation}\label{eq:fw-reciprocal-update}
  x_{k+1}=(1-\gamma_k)x_k+\gamma_ks_k,
  \qquad \gamma_k=\frac{\alpha}{k+\beta},
  \qquad s_k\in\operatorname{vert}(K)
  \quad\text{for all }k\in\mathbb N_0,
\end{equation}
starting from any $x_0\in K$ and keeping the oracle condition
\eqref{eq:affine-fw-oracle}.  For exact oracles, there is a constant
$C_{\mathrm{ex}}=C_{\mathrm{ex}}(K,L,x^*,\alpha,\beta)$ such that
\begin{equation}\label{eq:fw-reciprocal-exact}
  \norm{x_k-x^*}\leq
  \frac{C_{\mathrm{ex}}}{(k+1)^{\min\{\alpha,1\}}}
  \qquad\text{for every integer }k\geq1.
\end{equation}
More generally, for every $p>0$ there is a constant
$C_p=C_p(K,L,x^*,\alpha,\beta,p)$ such that, for every $D\geq0$, if
$0\leq\delta_k\leq D(k+1)^{-p}$ for all integers $k\geq1$, then
\begin{equation}\label{eq:fw-reciprocal-inexact}
  \norm{x_k-x^*}\leq
  \frac{C_p(1+D)}{(k+1)^{\min\{\alpha,1,p\}}}
  \qquad\text{for every integer }k\geq1.
\end{equation}
The constants are independent of $x_0$, the admissible vertex choices,
and the error sequence; $C_p$ is also independent of $D$.  No upper bound
on $\delta_0$ is required.
\end{proposition}

\begin{proof}[Proof of Proposition~\ref{prop:affine-fw-schedules}]
Use $U_*$ and $A_*$ from the preceding proof.  Since $0<\gamma_k\leq1$,
all iterates belong to $K$.  Start the normalization at index $1$, so that
the possible unit first step when $\beta=\alpha$ causes no division by
zero.  Define, for every integer $k\geq1$,
\begin{equation}\label{eq:fw-reciprocal-scaling}
  \rho_1:=1,\qquad
  \rho_{k+1}:=\frac{\rho_k}{1-\gamma_k},\qquad
  w_k:=\rho_{k+1}\gamma_k,\qquad
  z_k:=\rho_k(x_k-x^*).
\end{equation}
Here $\gamma_k<1$ for $k\geq1$, so every $\rho_k$ is positive.  The
update and the identity $\rho_{k+1}(1-\gamma_k)=\rho_k$ give
\[
  z_{k+1}=z_k+w_ku_k,\qquad u_k:=s_k-x^*\in U_*.
\]
By affinity and \eqref{eq:interior-orthogonality}, multiplying the oracle
inequality by $\rho_k$ gives
\[
  \ip{A_*z_k}{u_k}
  \leq\min_{u\in U_*}\ip{A_*z_k}{u}+\rho_k\delta_k.
\]
Let $d_K:=\operatorname{diam}K$.  For an integer $T\geq2$, put
\[
  \overline w_T:=\max_{1\leq k<T}w_k,\qquad
  B_T:=\max_{1\leq k<T}\rho_k\delta_k.
\]
The finite-horizon form of Corollary~\ref{cor:weighted}, applied from
index $1$ to index $T$, and $\norm{z_1}=\norm{x_1-x^*}\leq d_K$ imply
\begin{equation}\label{eq:fw-reciprocal-horizon}
  \norm{x_T-x^*}
  \leq\frac{C_{\mathrm w}}{\rho_T}\bigl(d_K+\overline w_T+B_T\bigr),
  \qquad C_{\mathrm w}=C_{\mathrm w}(T_K,U_*,A_*).
\end{equation}
In particular, the transformed weights and additive oracle errors need not be
bounded over the entire infinite sequence.

We next estimate the quantities in \eqref{eq:fw-reciprocal-horizon}.
There are positive constants $c_0,c_1,c_2$, depending only on
$\alpha,\beta$, such that, for every integer $k\geq1$,
\begin{equation}\label{eq:fw-reciprocal-growth}
  c_0(k+1)^\alpha\leq\rho_k\leq c_1(k+1)^\alpha,
  \qquad w_k\leq c_2(k+1)^{\alpha-1}.
\end{equation}
To verify the first two inequalities, take logarithms in
\eqref{eq:fw-reciprocal-scaling}:
\[
  \log\rho_k
  =\sum_{j=1}^{k-1}-\log\left(1-\frac{\alpha}{j+\beta}\right)
  =\alpha\log(k+1)+O(1),
\]
where the last term is bounded uniformly in $k$.  Indeed, with
$q:=\alpha/(1+\beta)<1$, the identity
\[
  0\leq-\log(1-t)-t
  =\int_0^t\frac{s}{1-s}\,ds
  \leq\frac{t^2}{2(1-q)}
  \qquad\text{for all }0\leq t\leq q
\]
bounds the sum of the remainders by a constant multiple of
$\sum_{j\geq1}(j+\beta)^{-2}<\infty$.  Integral comparison also bounds
$\sum_{j=1}^{k-1}(j+\beta)^{-1}-\log(k+1)$ uniformly in $k$.
Exponentiating proves the bounds on $\rho_k$.  The bound on $w_k$ follows
from $w_k=\alpha\rho_{k+1}/(k+\beta)$.

It follows from \eqref{eq:fw-reciprocal-growth} that
\[
  \overline w_T\leq c_2T^{\max\{\alpha-1,0\}}.
\]
For exact oracles, $B_T=0$, so \eqref{eq:fw-reciprocal-horizon} gives
\eqref{eq:fw-reciprocal-exact}.  Under the error hypothesis of the
proposition, one also has
\[
  B_T\leq c_1D T^{\max\{\alpha-p,0\}}.
\]
Substitution in \eqref{eq:fw-reciprocal-horizon} proves
\eqref{eq:fw-reciprocal-inexact} for $T\geq2$.  Increasing the constants
to cover $k=1$ uses only $\norm{x_1-x^*}\leq d_K$.  This also explains
why no bound on the initial oracle tolerance is needed.
\end{proof}

For example, for $\gamma_k=2/(k+2)$ the factors above are
$\rho_k=k(k+1)/2$ and $w_k=k+1$ for $k\geq1$.  Thus $\overline w_T=T$ and,
for exact oracles, \eqref{eq:fw-reciprocal-horizon} gives
\[
  \norm{x_T-x^*}\leq
  \frac{2C_{\mathrm w}(d_K+T)}{T(T+1)}
  \qquad\text{for every integer }T\geq2.
\]
This illustrates how a finite-horizon weighted bound gives an
$O(T^{-1})$ norm rate even though the transformed weights grow with $T$.

\paragraph{A saddle-point illustration.}\label{ex:saddle}
Let $X\subset\R^p$ and $Y\subset\R^q$ be nonempty compact polytopes.
Consider $\min_{x\in X}\max_{y\in Y}\mathcal L(x,y)$ with
\begin{equation}\label{eq:quadratic-saddle-model}
 \mathcal L(x,y)
 =\tfrac12x^\top Qx+x^\top My-\tfrac12y^\top Ry+b^\top x-c^\top y,
\end{equation}
where $Q$ and $R$ are symmetric positive definite, $M\in\R^{p\times q}$
is arbitrary, and $b\in\R^p$, $c\in\R^q$.  The convex--concave
first-order conditions identify its saddle points with VI solutions
on $K=X\times Y$; see Rockafellar \cite[Section~27]{Rockafellar1970}.
Writing $\xi=(x,y)$, the affine operator is
\begin{equation}\label{eq:saddle-operator}
 \Phi(\xi)=(\nabla_x\mathcal L,-\nabla_y\mathcal L)
 =L\xi+\begin{pmatrix}b\\c\end{pmatrix},
 \qquad
 L=\begin{pmatrix}Q&M\\-M^\top&R\end{pmatrix}.
\end{equation}
Its symmetric part is $\operatorname{diag}(Q,R)$, so $L$ is coercive
with $\mu=\min\{\lambda_{\min}(Q),\lambda_{\min}(R)\}$, independently
of $M$.  If the saddle point $\xi^*$ belongs to
$\operatorname{relint}(X\times Y)$, Theorem~\ref{thm:app-affine-fw}
therefore gives
\[
 \norm{\xi_k-\xi^*}\leq\frac{C(1+\overline B)}{k}
 \qquad\text{for all integers }k\geq1
\]
for the harmonic vertex iteration
\eqref{eq:affine-fw}--\eqref{eq:affine-fw-oracle}, provided
$\overline B=\sup_{k\geq1}k\delta_k<\infty$.
The LMO separates into one linear minimization over each factor;
it does not solve the full quadratic best-response problems.
The entire bilinear interaction lies in the skew-symmetric part of $L$:
replacing $L$ by its symmetric part removes that interaction.
No smallness condition on $M$ is required, although the rate constant
may depend on it.

\subsection{Oblique relaxation and finite coordinate correction}
\label{subsec:oblique-relaxation}

We now consider linear inequalities and an additive correction rule,
without the averaging step used in Frank--Wolfe.  A fixed matrix maps
the normal of the selected constraint to its correction direction.
We first show that approximate most-violated selection produces bounded
iterates, even for inconsistent systems.  We then use progress in a
fixed linear functional to deduce finite termination, and apply this
argument to the coordinate rule of
Corollary~\ref{cor:coordinate-termination}.

Consider the finite system
\begin{equation}\label{eq:oblique-system}
 \ip{a_i}{x}\geq b_i\qquad\text{for all }i=1,\ldots,m,
\end{equation}
where $m,n\geq1$, $a_i\in\R^n$, and $b_i\in\R$.
Let $P\in\R^{n\times n}$ have positive-definite symmetric part.
For $B\geq0$, start at $x_0\in\R^n$ and stop if all inequalities hold.
At each index $k$ before stopping, choose $i_k$ and set
\begin{equation}\label{eq:oblique-rule}
 \ip{a_{i_k}}{x_k}-b_{i_k}
 \leq\min_{1\leq i\leq m}(\ip{a_i}{x_k}-b_i)+B,
 \qquad x_{k+1}=x_k+Pa_{i_k}.
\end{equation}
The term \emph{oblique} indicates that $Pa_i$ need not be parallel to
$a_i$.  The algorithm is a correction procedure, not a projection onto
the selected halfspace.  After stopping, keep the state constant.

Amaldi and Hauser \cite[Theorem~7.1]{AmaldiHauser2005} established
boundedness for normal-direction corrections; a change of coordinates
also covers symmetric positive-definite $P$.
The following proposition allows nonsymmetric coercive $P$, but
requires approximate most-violated selection rather than arbitrary
selection of a violated inequality.  Section~\ref{sec:related}
compares these assumptions in more detail.

\begin{proposition}[Bounded oblique relaxation]
\label{prop:oblique-relaxation}
There is a constant $C_P=C_P(a_1,\ldots,a_m,P)\geq1$ such that every
procedure \eqref{eq:oblique-rule} satisfies
\begin{equation}\label{eq:oblique-bound}
 \sup_{k\geq0}\norm{x_k}\leq M,
 \qquad
 M:=C_P(1+\norm{x_0}+B+\Delta_b),
\end{equation}
where
\[
 \Delta_b:=\max\{0,b_1,\ldots,b_m\}
            -\min\{0,b_1,\ldots,b_m\}.
\]
No feasibility assumption on \eqref{eq:oblique-system} is required.
If, in addition, there exist $c\in\R^n$ and $\gamma>0$ such that
\begin{equation}\label{eq:oblique-separator}
 \ip{c}{Pa_i}\geq\gamma\qquad\text{for all }i=1,\ldots,m,
\end{equation}
then the procedure terminates at a feasible point after a finite number
$T$ of corrections satisfying
\begin{equation}\label{eq:oblique-count}
 T\leq\frac{\norm c}{\gamma}(M+\norm{x_0}).
\end{equation}
\end{proposition}

\begin{proof}
Set $a_0=0$, $b_0=0$, and
\[
 U_P:=\{Pa_i:0\leq i\leq m\},\qquad A_P:=P^{-\mathsf T}.
\]
The matrix $P$ is invertible by coercivity.  Moreover,
\[
 \frac{A_P+A_P^{\mathsf T}}2
 =P^{-\mathsf T}\frac{P+P^{\mathsf T}}2P^{-1}\succ0,
 \qquad \ip{A_Px}{Pa_i}=\ip{a_i}{x}.
\]
The set $U_P$ is finite and contains zero.  Before stopping the minimum
residual is negative, so adding index $0$, whose residual is zero,
does not change that minimum.  After stopping, select index $0$;
its residual is then a minimum, and the zero update leaves all
inequalities unchanged.  Thus, at every index $k\geq0$, the selected
index $i_k\in\{0,\ldots,m\}$ satisfies
\begin{align*}
 \ip{A_Px_k}{Pa_{i_k}}
 &=\ip{a_{i_k}}{x_k}-b_{i_k}+b_{i_k}\\
 &\leq\min_{0\leq j\leq m}(\ip{a_j}{x_k}-b_j)+B+b_{i_k}\\
 &\leq\min_{0\leq j\leq m}\ip{a_j}{x_k}+B+\Delta_b\\
 &=\min_{u\in U_P}\ip{A_Px_k}{u}+B+\Delta_b.
\end{align*}
The second inequality uses $b_{i_k}\leq\max_j b_j$ and
$-b_j\leq-\min_j b_j$.  Theorem~\ref{thm:main}, applied to
$(\R^n,U_P,A_P)$, proves \eqref{eq:oblique-bound}.

Suppose now that \eqref{eq:oblique-separator} holds.  For every integer
$N$ for which the procedure performs at least $N$ corrections,
\[
 N\gamma\leq\sum_{k=0}^{N-1}\ip{c}{Pa_{i_k}}
 =\ip{c}{x_N-x_0}\leq\norm c(M+\norm{x_0}).
\]
An infinite run would satisfy this inequality for every $N$, a
contradiction.  Hence the procedure stops, its output satisfies
\eqref{eq:oblique-system}, and taking $N=T$ proves
\eqref{eq:oblique-count}.
\end{proof}

The separator condition is sufficient, not necessary, for termination.
It implies strict feasibility: for $x=tP^{\mathsf T}c$ and sufficiently
large $t>0$, one has $\ip{a_i}{x}\geq t\gamma>b_i$ for every $i$.
No point satisfying all constraints as equalities is needed.
Feasibility alone does not ensure termination of this correction rule.
For $P=1$, the system $x\geq0$, $-x\geq-1/2$, and $x_0=3/4$, exact
most-violated selection alternates between $3/4$ and $-1/4$ forever.

\paragraph{Coordinate corrections.}
For $Gx\geq b$, write $a_i=G^{\mathsf T}e_i$ and choose
$P=G^{-\mathsf T}$.  Then
\[
 Pa_i=e_i,\qquad
 \frac{P+P^{\mathsf T}}2
 =G^{-1}\frac{G+G^{\mathsf T}}2G^{-\mathsf T}\succ0.
\]
The separator is $c=\mathbf1=(1,\ldots,1)^{\mathsf T}$ with $\gamma=1$.
This identifies \eqref{eq:coordinate-correction} as an oblique
relaxation.  The inverse is used only to explain the reduction: the
algorithm itself adds $e_{i_k}$ to $x_k$ and the corresponding column
of $G$ to the residual $Gx_k-b$.  For the sharper bound in terms of
$\norm{x_0-G^{-1}b}$, translate the equality solution to zero.

\begin{proof}[Proof of Corollary~\ref{cor:coordinate-termination}]
Coercivity makes $G$ invertible.  Put $z_k=x_k-G^{-1}b$ and
$U_c=\{0,e_1,\ldots,e_n\}$.  Before stopping,
\[
 Gz_k=Gx_k-b,\qquad
 \ip{Gz_k}{e_{i_k}}
 \leq\min_{u\in U_c}\ip{Gz_k}{u}+B,
\]
because the minimum residual is negative.  If the procedure stops,
extend it by zero updates; zero then minimizes the transformed score
at every subsequent state.  Theorem~\ref{thm:main} gives a constant
$C=C(\R^n,U_c,G)=C(n,G)\geq1$ such that
\begin{equation}\label{eq:coordinate-state-bound}
 \sup_{k\geq0}\norm{x_k-G^{-1}b}
 =\sup_{k\geq0}\norm{z_k}\leq C(1+\rho+B).
\end{equation}
After $N$ corrections,
\[
 N=\ip{\mathbf1}{z_N-z_0}
 \leq\sqrt n(\norm{z_N}+\rho)
 \leq\sqrt n\bigl[C(1+B)+(C+1)\rho\bigr].
\]
Thus an infinite run is impossible.  Taking $N=T$ proves
\eqref{eq:coordinate-count}; \eqref{eq:coordinate-state-bound} also
bounds the distance of the feasible output from $G^{-1}b$.
\end{proof}

Here $0$ is a vertex of $\conv U_c$, rather than a relative-interior
point.  The boundary case of the cycling theorem therefore gives a
finite-termination conclusion different from the averaging estimate
used in the Frank--Wolfe application.

For coordinate corrections, uniformly bounded residual errors can
instead be allowed without decay.  Suppose that, at each index before
stopping, every available residual differs from the corresponding
component of $Gx_k-b$ by at most $\varepsilon\geq0$, and the selected
index minimizes the available residuals up to an error $\eta\geq0$.
The same score comparison from Appendix~\ref{app:oracle-errors} yields
the selection condition of Corollary~\ref{cor:coordinate-termination}
with $B=\eta+2\varepsilon$.  Hence finite termination is preserved.
The stopping test must still use the true inequalities $Gx_k\geq b$;
no decay of $\varepsilon$ or $\eta$ is required.

\paragraph{Bounded positive step lengths.}
Keep the selection and stopping rules in
Corollary~\ref{cor:coordinate-termination}, but replace the update by
$x_{k+1}=x_k+w_ke_{i_k}$, where
$0<w_{\min}\leq w_k\leq\overline w$ at every index before stopping.
Corollary~\ref{cor:weighted} gives
\[
 \sup_{k\geq0}\norm{z_k}
 \leq C_{\mathrm w}(\overline w+\rho+B).
\]
After $N$ corrections, $Nw_{\min}\leq\sum_{k<N}w_k
=\ip{\mathbf1}{z_N-z_0}$.  The same argument proves finite termination
and the bound
\[
 T\leq\frac{\sqrt n}{w_{\min}}
       \bigl[C_{\mathrm w}(\overline w+\rho+B)+\rho\bigr].
\]
More generally, finite termination follows whenever the cumulative
weight of an infinite run would diverge, because boundedness of $z_N$
bounds $\sum_{k<N}w_k$.  Without such a condition, boundedness alone
does not force termination.  For instance, $G=1$, $b=1$, $x_0=0$,
and $w_k=2^{-k-2}$ give an infinite exact run with $x_k<1/2$.

\paragraph{The role of nonsymmetry and coercivity.}
For
\[
 G_\kappa=\begin{pmatrix}1&\kappa\\-\kappa&1\end{pmatrix},
 \qquad \kappa>1,
\]
the symmetric part is the identity, and
Corollary~\ref{cor:coordinate-termination} applies for every $\kappa$.
The comparison matrix with diagonal entries $1$ and off-diagonal
entries $-\kappa$ has a negative eigenvalue, so $G_\kappa$ is not an
$H$-matrix of the type used by Frommer and Szyld
\cite[Section~4]{FrommerSzyld2023}.  The corollary does not require
diagonal dominance.  If coercivity is dropped, even a feasible system
can yield an unbounded run: for
$G=\left(\begin{smallmatrix}0&-1\\1&0\end{smallmatrix}\right)$,
$b=0$, and $x_0=(0,1)$, the exact rule gives $x_k=(k,1)$ for every
$k\geq0$, since the first residual is always $-1$ and the second is $k$.

\section{Limitations, Generalizations and Open Problems}\label{sec:limitations}

We first record two generalizations that follow directly from
Theorem~\ref{thm:main}, together with observations on the necessity of its
assumptions and the dependence of its bound on the data.  We then discuss
the limitations of the applications and questions that remain open.

\paragraph{Sublinear perturbations of the score.}
In \eqref{eq:trajectory}, replace $Az_k$ by $F(z_k)=Az_k+r(z_k)$,
where $r\colon E\to E$ is locally bounded and satisfies
\[
  \lim_{\norm z\to\infty}\frac{\norm{r(z)}}{\norm z}=0.
\]
Keep the additive update rule and the assumptions on $E$, $U$, and $A$.
Then there is a constant $C_F=C_F(E,U,A,r)$ such that, for every
$B\geq0$, every trajectory with this modified selection rule satisfies
\[
  \sup_{k\geq0}\norm{z_k}\leq C_F(1+\norm{z_0}+B).
\]
To prove this, let $C$ be the constant in Theorem~\ref{thm:main}.
If $R_U=0$, the trajectory is constant.  Otherwise, choose
$\varepsilon>0$ such that $2CR_U\varepsilon\leq1/2$.
Local boundedness and the sublinear growth assumption give a finite
$M_\varepsilon\geq0$ with
$\norm{r(z)}\leq M_\varepsilon+\varepsilon\norm z$ for all $z\in E$.
Comparison of the perturbed and unperturbed scores shows that $u_k$ is
an approximate minimizer for $Az_k$ with error at most
$B+2R_U(M_\varepsilon+\varepsilon\norm{z_k})$.
For an integer $N\geq1$, put
$M_N=\max_{0\leq k\leq N}\norm{z_k}$.
The finite prefix is thus approximate with the constant error
$B+2R_U(M_\varepsilon+\varepsilon M_N)$; it can be extended to an
infinite trajectory with this error by choosing exact minimizers thereafter.
Theorem~\ref{thm:main} yields
\[
  M_N\leq C(1+\norm{z_0}+B+2R_U M_\varepsilon)
       +2CR_U\varepsilon M_N.
\]
Absorbing the last term and letting $N\to\infty$ gives the claimed
bound with $C_F=2C(1+2R_U M_\varepsilon)$.
This consequence concerns sublinear perturbations of the additive score
at infinity; it does not establish the nonlinear Frank--Wolfe extension
discussed below.

\paragraph{Finite update sets in Hilbert spaces.}
The same form of bound holds when $E$ is replaced by a real Hilbert space
$\mathcal H$, provided $U$ remains finite and $A\colon\mathcal H\to\mathcal H$
is bounded, linear, and coercive.  Indeed, let $L=\operatorname{span}U$
and let $P_L$ denote the orthogonal projection onto $L$.
Write $z_k=p_k+q$, where $p_k=P_Lz_k$ and
$q=(\Id_{\mathcal H}-P_L)z_0$ is constant.
On the finite-dimensional space $L$, the operator
$A_L=P_LA|_L$ is coercive, and the oracle score is
\[
  \ip{Az_k}{u}=\ip{A_Lp_k+P_LAq}{u}
  \qquad\text{for all }u\in U.
\]
Thus $(p_k)$ is an approximate trajectory for $(L,U,A_L)$ with error
at most $B+2R_U\norm A\norm q$.
Applying Theorem~\ref{thm:main} and using
$\norm{p_0},\norm q\leq\norm{z_0}$ gives a bound of the same form,
with a constant independent of the initial state, $B$, and the admissible
selections.  If $L=\{0\}$, the conclusion is immediate.
Gelfand et al.\ \cite{GelfandEtAl2010} already note the Hilbert-space
extension for the classical cycling theorem.  The observation here
extends the robust nonsymmetric statement and applies, in particular,
to finite candidate sets in a reproducing-kernel Hilbert space.

\paragraph{Necessary assumptions and tightness of the bound.}
The condition $0\in\conv U$ is necessary for boundedness under any
selection rule.  If it fails, strict separation, as in Rockafellar
\cite[Section~11]{Rockafellar1970}, gives $c\in E$ and $\gamma>0$ such that
$\ip{c}{u}\geq\gamma$ for all $u\in U$.  Every $U$-valued trajectory
then satisfies
\[
  \ip{c}{z_k}\geq\ip{c}{z_0}+k\gamma
  \qquad\text{for all }k\in\mathbb N_0,
\]
and is unbounded, independently of the oracle.
The linear dependence on $B$ is also necessary.  For $E=\R$, $A=1$,
$U=\{-1,0,1\}$, and $B\geq2$, the trajectory
$z_k=\min\{k,\lfloor B/2\rfloor\}$ starts at zero and is
$B$-approximate: the update $+1$ has score error $2z_k$ while it is
used, and the subsequent update $0$ has score error $z_k\leq B$.
Its supremum is $\lfloor B/2\rfloor$, so no bound that grows sublinearly
in $B$ can hold uniformly over admissible trajectories.
Likewise, the supremum is at least $\norm{z_0}$ because it includes
the initial state.

\paragraph{Explicit constants.}
Theorem~\ref{thm:main} establishes the existence of $C(E,U,A)$.
Its compactness argument and finite covering of the degenerate directions
are qualitative.  One route to an effective estimate would be to quantify
the geometry of the normal fan of $\conv U$, including positive score
gaps and angular separations between faces.
The explicit bound in Theorem~\ref{thm:explicit-symmetric} instead uses
symmetry and a relative inradius.  An explicit nonsymmetric counterpart
remains to be derived.  In particular, for fixed $E$, $U$, and symmetric
part $S=(A+A^\top)/2\succ0$, it is natural to ask how a bound depends on
\[
 \kappa:=\frac{\norm{(A-A^\top)/2}}{\lambda_{\min}(S)}.
\]
The results impose no upper bound on this ratio, but do not give a
constant uniform as the ratio increases.

\paragraph{Finite output sets and oracle accuracy.}
Theorem~\ref{thm:main} requires a fixed finite update set, not necessarily
a set of vertices.  Corollary~\ref{cor:weighted} additionally permits
bounded nonnegative multiples of these updates, but its oracle still
selects a direction from the fixed finite set.
The vertex-returning Frank--Wolfe oracle and the finite set of
coordinate corrections meet this requirement.
An oracle that returns arbitrary points of a minimizing face can instead
produce infinitely many increments.  For an approximate oracle,
optimality of a convex combination controls only the weighted average
of its vertexwise errors; tiny positive weights can remain on badly
suboptimal vertices.  Such approximate oracles are not covered by the
finite-output theorem, and a robust extension would need additional
hypotheses.

A fixed additive error $B$ is allowed in the additive recursion.
For the harmonic Frank--Wolfe recursion, the original score error
$\delta_k$ becomes $k\delta_k$; the condition
$\sup_{k\geq1}k\delta_k<\infty$ gives the $O(k^{-1})$ norm rate.
Proposition~\ref{prop:affine-fw-schedules} additionally covers steps
$\alpha/(k+\beta)$ and errors $\delta_k=O(k^{-p})$ for $p>0$, with
norm rate $O(k^{-\min\{\alpha,1,p\}})$.  A fixed positive tolerance
in the original Frank--Wolfe score does not imply convergence to the exact
VI solution.  In the corresponding additive variables, the transformed
error grows with the state-normalization factor, rather than remaining
a fixed $B$.

\paragraph{Scope and complexity of coordinate correction.}
Corollary~\ref{cor:coordinate-termination} concerns square matrices
with positive-definite symmetric part.  Such a matrix is invertible,
so $Gx\geq b$ is always feasible.  The corollary is therefore a
termination guarantee for a specific update rule, not a method for
deciding feasibility of arbitrary linear systems.  The output need
not solve $Gx=b$, nor satisfy a complementarity condition.
Proposition~\ref{prop:oblique-relaxation} permits inconsistent systems,
but then asserts boundedness only, not finite detection of inconsistency.

The correction-count bound depends on the nonconstructive constant of
Theorem~\ref{thm:main}.  The explicit relative-inradius estimate in
Theorem~\ref{thm:explicit-symmetric} does not apply to
$U_c=\{0,e_1,\ldots,e_n\}$, even for symmetric $G$, because zero is
not in the relative interior of its convex hull.  For symmetric $G$,
Amaldi and Hauser's general bound \cite{AmaldiHauser2005} remains
available.  The coordinate update and the column update of the residual
require no inverse or projection, but computing a most-violated index
also has a cost.  Deriving effective nonsymmetric correction bounds
and comparing total runtime with established relaxation methods remain
open tasks; no computational-complexity improvement is asserted here.
Bounded selection errors preserve exact feasibility only with the
stated exact stopping test.

\Needspace{17\baselineskip}
\begin{samepage}
\section*{AI usage}
Following the
\emph{Leiden Declaration on Artificial Intelligence and Mathematics}
by Alper et al.\ \cite{AlperEtAl2026}, the author discloses discussions
with OpenAI's ChatGPT and Anthropic's Claude and the use of OpenAI's Codex.
The research originated in the question of last-iterate convergence of
Frank--Wolfe iterations for affine variational inequalities on polytopes.
ChatGPT -- 5.6 Sol suggested a connection with perceptron cycling theorems, whose
existing forms did not resolve the original question.  It then proposed
an initial proof sketch of a more general perceptron cycling theorem based on
induction on dimension and face reduction.  The author subsequently
corrected errors in the argumentation and developed this sketch further including 
the new Lyapunov approach for trajectory segments.  ChatGPT and Codex also assisted
with the further development and checking of mathematical arguments,
the write-up and literature searches.  Claude -- Faible 5.1 assisted with
proofreading, stylistic suggestions, and checking selected mathematical
claims.  The author takes sole responsibility for the proofs, the
final text, attribution, and the accuracy and completeness of references.
\end{samepage}

\appendix
\section{Basic Theorems in Calculus}

For completeness, this appendix records the precise finite-dimensional
analytic results invoked above.  All intervals below are compact, and all
Euclidean spaces are finite-dimensional.

\begin{theorem}[Arzel\`a--Ascoli, finite-dimensional form;
{Rudin \cite[Theorem~7.25]{RudinPMA}}]
\label{thm:appendix-aa}
Let $I=[a,b]$, let $E$ be a finite-dimensional normed space, and let
$(f_n)$ be a sequence of functions from $I$ to $E$.  Assume that $(f_n)$ is
uniformly bounded and equicontinuous; that is,
\[
  \sup_n\sup_{t\in I}\norm{f_n(t)}<\infty
\]
and, for every $\varepsilon>0$, there is a $\delta>0$ such that, for all
$n$ and all $s,t\in I$,
\[
  |s-t|<\delta
  \quad\Longrightarrow\quad
  \norm{f_n(s)-f_n(t)}<\varepsilon.
\]
Then $(f_n)$ has a subsequence that converges uniformly on $I$ to a
continuous function $f\colon I\to E$.  In particular, a uniformly bounded
family with a common Lipschitz constant satisfies the hypotheses.
\end{theorem}

\par\medskip
\begin{remark}[Why the finite-dimensional hypothesis suffices]
Choose a countable dense subset of $I$.  Bounded subsets of $E$ have compact
closure by the finite-dimensional Heine--Borel theorem in Rudin
\cite[Theorem~2.41]{RudinPMA}, so a diagonal argument produces a subsequence
converging at every point of that dense subset.  Equicontinuity makes this
subsequence uniformly Cauchy on $I$, hence uniformly convergent.  This is
the standard proof of the version used in Lemma~\ref{lem:limit}.
\end{remark}

\begin{theorem}[Lipschitz paths, Rademacher, and integration]
\label{thm:appendix-lipschitz}
If $f\colon[a,b]\to E$ is Lipschitz, then $f$ is absolutely continuous and
differentiable for almost every $t\in[a,b]$.  Its derivative is essentially
bounded by every Lipschitz constant of $f$, and, for all $s,t\in[a,b]$
with $s\leq t$,
\[
  f(t)-f(s)=\int_s^t f'(r)\,dr.
\]
Consequently, if $h\colon[a,b]\to\R$ is Lipschitz and
$h'(t)\leq-c$ for almost every $t$, then, for all $s,t\in[a,b]$
with $s\leq t$,
\[
  h(t)-h(s)\leq-c(t-s).
\]
The almost-everywhere differentiability assertion is the one-dimensional
case of Rademacher's theorem; see Evans and Gariepy
\cite[Section~3.1]{EvansGariepy}.  The absolute-continuity and integral
representation statements are the fundamental theorem of calculus for
absolutely continuous functions; see Rudin
\cite[Theorem~7.20]{RudinRCA}.  The vector-valued form follows by applying
the scalar theorem to finitely many coordinates.
\end{theorem}

\begin{theorem}[Lebesgue density theorem;
{Folland \cite[Theorem~3.21, applied to indicator functions]{Folland}}]
\label{thm:appendix-density}
If $M\subset\R$ is Lebesgue measurable, then almost every $t\in M$ is a
density point of $M$, meaning
\[
  \lim_{r\downarrow0}
  \frac{|M\cap(t-r,t+r)|}{2r}=1.
\]
In particular, if an interval is partitioned into finitely many measurable
sets $M_1,\dots,M_N$, then for almost every point $t$ of the interval, $t$
is a density point of the unique member of the partition that contains it.
\end{theorem}

\begin{lemma}[Derivative of a finite maximum]
\label{lem:appendix-finite-max}
Let $f_1,\dots,f_N$ be real-valued functions defined near $t_0$.  Suppose
that every $f_i$ is differentiable at $t_0$ and that, for some $a,c\in\R$,
\[
  f_i(t_0)=a,
  \qquad
  f_i'(t_0)=c
  \qquad\text{for all }i=1,\dots,N.
\]
Then $f:=\max_{1\leq i\leq N}f_i$ is differentiable at $t_0$ and
$f'(t_0)=c$.
\end{lemma}

\begin{proof}
For each $i$,
\[
  f_i(t_0+h)=a+ch+r_i(h),
  \qquad \lim_{h\to0}\frac{r_i(h)}h=0.
\]
Since there are only finitely many indices,
$\lim_{h\to0}\max_i|r_i(h)|/|h|=0$.  Taking the maximum over $i$ therefore gives
\[
  f(t_0+h)=a+ch+o(|h|),
\]
which is the asserted differentiability.
\end{proof}

\clearpage
\section{Comparison of explicit bounds}\label{app:bound-comparison}

The relative-inradius estimate can remain uniform when a direct
application of the relaxation bound deteriorates.  For $0<\varepsilon<1/2$,
consider
\[
 E=\R,\qquad A=1,\qquad
 U_\varepsilon=\{-1,1,1-\varepsilon\},\qquad B=1,\qquad z_0=0.
\]
Here $\conv U_\varepsilon=[-1,1]$, so $R=r=1$ and
Theorem~\ref{thm:explicit-symmetric} gives
\[
 \sup_{k\geq0}|z_k|\leq1+\frac{2+1}{2}=\frac52,
\]
independently of $\varepsilon$.

Every admissible increment satisfies $z_ku_k\leq1$, since the minimum
score is nonpositive.  Apply Amaldi and Hauser's Theorem~7.1
\cite{AmaldiHauser2005} with constraint matrix
$D_\varepsilon=(-1,1,1-\varepsilon)$, right-hand side
$b=(1,1,1)^{\mathsf T}$, initial state zero, and all step coefficients
equal to one.  The homogenized matrix in their formula is
\[
 \widetilde D_\varepsilon
 =\begin{pmatrix}
    -1&1&1-\varepsilon&0\\
    -1&-1&-1&1
   \end{pmatrix}.
\]
Their condition number $\kappa(\widetilde D_\varepsilon)$ is the
reciprocal of the smallest positive volume generated by a linearly
independent subset of its columns \cite[Eq.~(3.1)]{AmaldiHauser2005}.
The smallest column norm is $1$, and the positive absolute determinants
of pairs of columns are $\varepsilon$, $1-\varepsilon$, $1$,
$2-\varepsilon$, and $2$.  Consequently,
\[
 \kappa(\widetilde D_\varepsilon)=\varepsilon^{-1},\qquad
 \operatorname{rank}(\widetilde D_\varepsilon)=2,\qquad
 \max_j\norm{\widetilde d_j}=\sqrt2,
\]
where $\widetilde d_j$ denotes the $j$th column.
Since the original ambient dimension is one, their estimate
\cite[Eq.~(7.1)]{AmaldiHauser2005} becomes
\[
 M_\varepsilon
 =2\max\left\{1,\,1+\sqrt2\left(\frac8\varepsilon+1\right)\right\}
 =2+2\sqrt2\left(\frac8\varepsilon+1\right).
\]
Thus $\lim_{\varepsilon\downarrow0}M_\varepsilon=\infty$, whereas
the relative-inradius bound stays equal to $5/2$.  The added update
$1-\varepsilon$ leaves the convex hull unchanged but makes two
homogenized columns nearly identical.  This comparison concerns the
displayed general formula under the specified application, not the best
bound obtainable by further specializing their theorem, and does not
assert that our estimate is smaller for every instance.

\section{Oracle error models}\label{app:oracle-errors}

This appendix makes explicit how imperfect information in the additive
recursion leads to the tolerance $B$ in Theorem~\ref{thm:main}.
We use the standing assumptions on $E$, $U$, and $A$ from that theorem.
The proposition below concerns the additive states $z_k$; the original
Frank--Wolfe score errors have the different scaling explained in
Subsection~\ref{subsec:affine-fw-app}.  The proof combines the standard
score comparisons of Freund and Grigas
\cite[Section~5.2, Proposition~5.1]{FreundGrigas2016} with the distance
traveled during a bounded delay.  Bounded-delay models are discussed by
Bertsekas and Tsitsiklis \cite{BertsekasTsitsiklis1989}, and quantization
models by Gray and Neuhoff \cite{GrayNeuhoff1998}.

\begin{proposition}[Combined oracle imperfections]
\label{prop:robustness-mechanisms}
Let $(z_k)$ be a $U$-valued trajectory, let
$R_U=\max_{u\in U}\norm u$.  Fix
$\tau\in\mathbb N_0$ and $\varepsilon,\eta,\nu\geq0$.  Suppose that, for
every $k\in\mathbb N_0$, there are an integer
$0\leq\tau_k\leq\min\{\tau,k\}$ and a perturbation $n_k\in E$ such that
\[
  \widetilde z_k:=z_{k-\tau_k}+n_k,
  \qquad \norm{n_k}\leq\nu.
\]
Assume also that, for every $k\in\mathbb N_0$, the available scores satisfy
\[
  \max_{u\in U}
  \left|
    \widehat\ell_k(u)-\ip{A\widetilde z_k}{u}
  \right|
  \leq\varepsilon
\]
and that the selected increment satisfies
\[
  \widehat\ell_k(u_k)
  \leq\min_{u\in U}\widehat\ell_k(u)+\eta.
\]
Then $(z_k)$ is a $B$-approximate trajectory with
\begin{equation}\label{eq:combined-oracle-error}
  B=\eta+2\varepsilon
    +2\norm A R_U(\tau R_U+\nu).
\end{equation}
In particular, the following cases are obtained by setting the unused error
parameters equal to zero:
\begin{enumerate}
\item Inaccurate or quantized scores at the current state give
$B=\eta+2\varepsilon$.
\item An $\eta$-optimal update computed at a state delayed by at most
$\tau$ iterations gives
\[
  B=\eta+2\tau\norm A R_U^2.
\]
\item An $\eta$-optimal update computed at a perturbed or quantized current
state with error at most $\nu$ gives
\[
  B=\eta+2\norm A R_U\nu.
\]
\end{enumerate}
Consequently, Theorem~\ref{thm:main} and the unit-weight case of
Corollary~\ref{cor:averaging} give a state bound and an $O(T^{-1})$
average-update bound, with a numerator depending at most linearly on the
displayed error parameters.
\end{proposition}

The factor two in the proof appears because the score error must be
controlled once for the selected update and once for a true minimizer.

\begin{proof}[Proof of Proposition~\ref{prop:robustness-mechanisms}]
Fix $k\in\mathbb N_0$ and $u\in U$.  The score accuracy and approximate
minimization assumptions give
\begin{align*}
 \ip{A\widetilde z_k}{u_k}
 &\leq\widehat\ell_k(u_k)+\varepsilon\\
 &\leq\widehat\ell_k(u)+\eta+\varepsilon\\
 &\leq\ip{A\widetilde z_k}{u}+\eta+2\varepsilon.
\end{align*}
Because the intervening increments belong to $U$,
\[
 \norm{z_k-z_{k-\tau_k}}
 \leq\tau_kR_U\leq\tau R_U.
\]
Consequently,
\[
 \norm{z_k-\widetilde z_k}
 \leq\tau R_U+\nu,
 \qquad \norm{u_k-u}\leq2R_U.
\]
It follows that
\begin{align*}
 \ip{Az_k}{u_k-u}
 &=\ip{A\widetilde z_k}{u_k-u}
   +\ip{A(z_k-\widetilde z_k)}{u_k-u}\\
 &\leq\eta+2\varepsilon
   +2\norm A R_U(\tau R_U+\nu).
\end{align*}
Taking the minimum over $u\in U$ proves
\eqref{eq:combined-oracle-error}.  The three displayed special cases follow
by setting the unused parameters equal to zero.  The final bounds follow
from Theorem~\ref{thm:main} and the unit-weight case of
Corollary~\ref{cor:averaging}.
\end{proof}

The parameter $\varepsilon$ covers quantized individual scores directly: it
is enough that every quantized score differs from its true value by at most
$\varepsilon$.  The parameter $\nu$ instead models a quantizer applied to the
state before all scores are formed.  No independence, unbiasedness, or
stochastic assumption is required in either formulation.

\end{document}

%% file: Proof_Overview_Combined-v2.tikz
\definecolor{proofink}{HTML}{263746}
\definecolor{proofmuted}{HTML}{67727D}
\definecolor{proofguide}{HTML}{A7B0B8}
\definecolor{proofgreen}{HTML}{2F7D49}
\definecolor{proofred}{HTML}{B33E46}
\definecolor{proofentrance}{HTML}{9A6024}
\definecolor{proofdirectionpale}{HTML}{F6E9CD}

\newcommand{\ProofOverviewCombined}{%
\begin{tikzpicture}[
  x=1cm,y=1cm,line cap=round,line join=round,
  font=\fontsize{8}{9.6}\selectfont,
  text=proofink,>=Latex,
  boundary/.style={draw=proofguide,line width=.65pt},
  intermediate/.style={draw=proofguide,dash pattern=on 3pt off 2.5pt,line width=.6pt},
  leader/.style={draw=proofmuted,line width=.5pt},
  caseone/.style={draw=proofgreen,line width=1.55pt},
  casetwo/.style={draw=proofred,line width=1.55pt},
  greendot/.style={circle,fill=proofgreen,draw=white,line width=.35pt,
    inner sep=0pt,minimum size=3.1pt},
  reddot/.style={rectangle,fill=proofred,draw=white,line width=.35pt,
    inner sep=0pt,minimum size=3.0pt},
  note/.style={font=\fontsize{8}{9.6}\selectfont,align=center}
]
  \path[use as bounding box] (-3.08,-3.42) rectangle (3.12,2.98);

  \path[fill=proofdirectionpale]
    (40:.233333) -- (40:2.8)
    arc[start angle=40,end angle=85,radius=2.8]
    -- (85:.233333)
    arc[start angle=85,end angle=40,radius=.233333] -- cycle;
  \draw[draw=proofentrance!45,densely dashed,line width=.45pt]
    (40:.233333) -- (40:2.8);
  \draw[draw=proofentrance!45,densely dashed,line width=.45pt]
    (85:.233333) -- (85:2.8);
  \draw[boundary] (0,0) circle[radius=2.8];
  \draw[intermediate] (0,0) circle[radius=.933333];
  \draw[boundary,fill=proofguide!12] (0,0) circle[radius=.233333];

  \node[note,text=proofentrance,text width=1.36cm] at (.93,2.16)
    {cone\\$\mathcal C(W)$};
  \node[anchor=east] at (-2.48,1.38) {$R$};
  \node[anchor=east] (middlelabel) at (-1.32,.66) {$R/3$};
  \draw[leader] (middlelabel.east) -- (145:.933333);
  \node[anchor=east] (innerlabel) at (-.34,.04) {$\varepsilon R$};
  \draw[leader] (innerlabel.east) -- (170:.233333);
  \fill[proofink] (0,0) circle[radius=.8pt];
  \node[font=\fontsize{7}{8.4}\selectfont,anchor=south west,inner sep=1.4pt]
    at (0,.025) {$0$};

  \draw[draw=proofgreen!45,line width=.7pt]
    (.025,.015) -- (.06,-.07) -- (-.03,-.22) -- (.08,-.32);
  \draw[draw=proofred!45,line width=.7pt]
    (-.02,.025) -- (-.07,-.055) -- (-.115,-.19) -- (-.18,-.28);
  \draw[draw=proofgreen!65,fill=white,line width=.55pt]
    (-.03,-.22) circle[radius=1.3pt];
  \node[rectangle,draw=proofred!65,fill=white,line width=.55pt,
    inner sep=0pt,minimum size=2.6pt] at (-.115,-.19) {};

  \coordinate (greenstart) at (.08,-.32);
  \coordinate (greenlast) at (1.52,1.12);
  \coordinate (entry) at (1.50,1.30);
  \coordinate (greentarget) at (55:2.86);
  \draw[caseone]
    (greenstart) -- (.13,-.50) -- (.35,-.60) -- (.48,-.50)
    -- (.65,-.64) -- (.85,-.48) -- (1.03,-.31) -- (1.14,-.15)
    -- (1.25,.05) -- (1.40,.24) -- (1.47,.40) -- (1.55,.63)
    -- (1.52,.85) -- (greenlast);
  \foreach \p in {(.35,-.60),(.85,-.48),(1.25,.05),(1.55,.63)}
    \node[greendot] at \p {};
  \node[greendot] at (greenstart) {};
  \node[greendot] at (greenlast) {};
  \draw[caseone,-{Latex[length=1.8mm,width=1.2mm]}]
    (1.03,-.31) -- (1.14,-.15);
  \node[anchor=west,text=proofgreen] at (.20,-.27) {$z_s$};
  \node[anchor=west,text=proofgreen,inner sep=1.5pt]
    at (1.65,1.12) {$z_\ell$};
  \draw[draw=proofgreen!50,line width=.75pt] (greenlast) -- (entry);
  \draw[draw=proofgreen!50,line width=.8pt,dash pattern=on 2.3pt off 2pt]
    (entry) -- (1.55,1.48) -- (1.44,1.62) -- (1.55,1.82)
    -- (1.61,2.00) -- (greentarget);
  \node[greendot,minimum size=5pt] at (entry) {};
  \node[circle,draw=proofgreen!50,fill=white,inner sep=0pt,minimum size=4pt]
    at (greentarget) {};
  \node[note,text width=2.05cm,inner sep=1.4pt] (entrylabel)
    at (-1.02,1.42) {first relevant\\entrance $z_e$\\$\norm{z_e}\ge R/3$};
  \draw[leader] (entrylabel.east) -- (1.14,1.56) -- (entry);
  \node[note,text=proofmuted,anchor=south west,inner sep=1pt]
    at (1.84,2.34) {target\\crossing};

  \coordinate (redstart) at (-.18,-.28);
  \coordinate (redlast) at (-2.01,-1.79);
  \coordinate (redtarget) at (-2.20,-1.86);
  \draw[casetwo]
    (redstart) -- (-.38,-.42) -- (-.50,-.32) -- (-.67,-.47)
    -- (-.92,-.52) -- (-1.02,-.69) -- (-1.21,-.78) -- (-1.36,-1.00)
    -- (-1.65,-1.05) -- (-1.72,-1.29) -- (-1.90,-1.38)
    -- (-1.99,-1.61) -- (redlast);
  \foreach \p in {(-.50,-.32),(-1.02,-.69),(-1.65,-1.05),(-1.99,-1.61)}
    \node[reddot] at \p {};
  \node[reddot] at (redstart) {};
  \node[reddot] at (redlast) {};
  \draw[casetwo,-{Latex[length=1.8mm,width=1.2mm]}]
    (-1.21,-.78) -- (-1.36,-1.00);
  \node[anchor=east,text=proofred] at (-.28,-.18) {$z_s$};
  \node[anchor=west,text=proofred,inner sep=1.5pt]
    at (-1.86,-1.66) {$z_\ell$};
  \draw[draw=proofred!50,line width=.75pt] (redlast) -- (redtarget);
  \node[reddot,minimum size=4.7pt] at (redtarget) {};
  \node[note,text width=2.9cm] (targetlabel) at (.20,-2.06)
    {target crossing\\$\norm{z_{k_R}}\ge R$};
  \draw[leader] (targetlabel.west) -- (-1.66,-2.06) -- (redtarget);

  \draw[caseone] (-1.85,-3.13) -- (-1.35,-3.13);
  \node[greendot] at (-1.60,-3.13) {};
  \node[anchor=west,text=proofgreen] at (-1.23,-3.13) {Case 1};
  \draw[casetwo] (.40,-3.13) -- (.90,-3.13);
  \node[reddot] at (.65,-3.13) {};
  \node[anchor=west,text=proofred] at (1.02,-3.13) {Case 2};
\end{tikzpicture}%
}

%% file: Cube_Three_Regimes-v2.tikz
\definecolor{cubeInk}{HTML}{263746}
\definecolor{cubeMuted}{HTML}{697783}
\definecolor{cubeGuide}{HTML}{B7C1C8}
\definecolor{cubeTeal}{HTML}{007D79}
\definecolor{cubeOrange}{HTML}{B86321}
\definecolor{cubeRed}{HTML}{B33E46}

\newcommand{\CubeTrajectoryPanel}[1]{%
\begin{tikzpicture}[
  x=1cm,y=1cm,line cap=round,line join=round,
  font=\fontsize{8}{9.6}\selectfont,text=cubeInk,
  >={Latex[length=1.35mm,width=.95mm]},
  cube axis/.style={draw=cubeMuted,line width=.45pt,->},
  cube grid/.style={draw=cubeGuide!35,line width=.25pt},
  cube state/.style={circle,draw=white,line width=.4pt,
    inner sep=0pt,minimum size=3.6pt}
]
  \path[use as bounding box] (-2.45,-2.35) rectangle (2.50,2.40);
  \ifnum#1=0
    \def\cubeAA{1}\def\cubeAB{0}\def\cubeBA{0}\def\cubeBB{1}
    \def\cubeN{2}\def\cubeScale{1}\def\cubeLimit{2}
    \def\cubeTicks{-2,-1,1,2}\def\cubeColor{cubeTeal}
  \else\ifnum#1=1
    \def\cubeAA{1}\def\cubeAB{-2}\def\cubeBA{2}\def\cubeBB{1}
    \def\cubeN{6}\def\cubeScale{1}\def\cubeLimit{2}
    \def\cubeTicks{-2,-1,1,2}\def\cubeColor{cubeOrange}
  \else
    \def\cubeAA{0}\def\cubeAB{-1}\def\cubeBA{1}\def\cubeBB{0}
    \def\cubeN{45}\def\cubeScale{.2}\def\cubeLimit{10}
    \def\cubeTicks{-10,-5,5,10}\def\cubeColor{cubeRed}
  \fi\fi
  \pgfmathsetmacro{\cubeAxisMax}{2.20/\cubeScale}
  \begin{scope}[x=\cubeScale cm,y=\cubeScale cm]
    \fill[cubeGuide!14] (-1,-1) rectangle (1,1);
    \draw[draw=cubeGuide,line width=.6pt] (-1,-1) rectangle (1,1);
    \foreach \t in \cubeTicks {
      \draw[cube grid] (\t,-\cubeLimit) -- (\t,\cubeLimit);
      \draw[cube grid] (-\cubeLimit,\t) -- (\cubeLimit,\t);
      \draw[draw=cubeMuted,line width=.4pt]
        (\t,{-0.035/\cubeScale}) -- (\t,{0.035/\cubeScale});
      \draw[draw=cubeMuted,line width=.4pt]
        ({-0.035/\cubeScale},\t) -- ({0.035/\cubeScale},\t);
      \node[anchor=north,inner sep=2pt,font=\fontsize{7}{8.4}\selectfont]
        at (\t,{-0.035/\cubeScale}) {$\t$};
      \node[anchor=east,inner sep=2pt,font=\fontsize{7}{8.4}\selectfont]
        at ({-0.035/\cubeScale},\t) {$\t$};
    }
    \draw[cube axis] (-\cubeAxisMax,0) -- (\cubeAxisMax,0);
    \draw[cube axis] (0,-\cubeAxisMax) -- (0,\cubeAxisMax);
    \node[anchor=south east,inner sep=1pt]
      at ({2.18/\cubeScale},{.09/\cubeScale}) {$z^{(1)}$};
    \node[anchor=north west,inner sep=1pt]
      at ({.09/\cubeScale},{2.18/\cubeScale}) {$z^{(2)}$};
    \foreach \a in {-1,1}{\foreach \b in {-1,1}{
      \node[rectangle,fill=cubeInk,inner sep=0pt,minimum size=2.8pt]
        at (\a,\b) {};
    }}
    \coordinate (cubeZ0) at (.75,.4);
    \foreach \k [remember=\nx as \px (initially 15),
                 remember=\ny as \py (initially 8)] in {1,...,\cubeN}{
      \pgfmathtruncatemacro{\ax}{\cubeAA*\px+\cubeAB*\py}
      \pgfmathtruncatemacro{\ay}{\cubeBA*\px+\cubeBB*\py}
      \pgfmathtruncatemacro{\nx}{\px+(\ax>0 ? -20 : 20)}
      \pgfmathtruncatemacro{\ny}{\py+(\ay>0 ? -20 : 20)}
      \pgfmathtruncatemacro{\prev}{\k-1}
      \coordinate (cubeZ\k) at (\nx/20,\ny/20);
      \ifnum#1=1
        \draw[draw=\cubeColor,line width=.9pt,->,
          shorten <=2.2pt,shorten >=2.2pt]
          (cubeZ\prev) -- (cubeZ\k);
      \fi
      \ifnum#1=2
        \draw[draw=\cubeColor,line width=.7pt]
          (cubeZ\prev) -- (cubeZ\k);
        \fill[\cubeColor] (cubeZ\k) circle[radius=.65pt];
      \fi
    }
    \ifnum#1=0
      \draw[draw=\cubeColor,line width=.95pt,<->,
        shorten <=2.2pt,shorten >=2.2pt] (cubeZ0) -- (cubeZ1);
      \node[cube state,fill=\cubeColor] at (cubeZ1) {};
      \node[text=\cubeColor,anchor=north,inner sep=2pt]
        at (-.25,-.70) {$z_1$};
      \node[text=\cubeColor,anchor=south,inner sep=2pt]
        at (.75,.49) {$z_0=z_2$};
    \else\ifnum#1=1
      \foreach \k in {1,...,5}{
        \node[cube state,fill=\cubeColor] at (cubeZ\k) {};
      }
      \node[text=\cubeColor,anchor=south west,inner sep=2pt]
        at (.78,.48) {$z_0=z_6$};
    \else
      \foreach \k in {5,12,23,36,44}{
        \pgfmathtruncatemacro{\prev}{\k-1}
        \draw[draw=\cubeColor,line width=.7pt,->]
          (cubeZ\prev) -- (cubeZ\k);
      }
      \node[cube state,fill=\cubeColor,minimum size=4pt] at (cubeZ45) {};
      \node[text=\cubeColor,anchor=north,inner sep=1pt] (cubeEndLabel)
        at (6,-3) {$z_{45}$};
      \draw[draw=\cubeColor!65,line width=.45pt]
        (cubeEndLabel.north) -- (8,-2) -- (cubeZ45);
      \node[text=\cubeColor,anchor=south east,inner sep=1pt] (cubeStartLabel)
        at (-1.6,1.9) {$z_0$};
      \draw[draw=\cubeColor!65,line width=.4pt]
        (cubeStartLabel.east) -- (cubeZ0);
    \fi\fi
    \node[circle,fill=white,draw=\cubeColor,line width=.75pt,
      inner sep=0pt,minimum size=4.2pt] at (cubeZ0) {};
  \end{scope}
\end{tikzpicture}%
}